\documentclass{amsart}

\usepackage{amssymb}
\newcommand{\cz}[2]{C_{#1}(#2)}
\newcommand{\cc}[1]{C(#1)}
\newcommand{\cstar}[2]{C^{*}_{#1}(#2)}
\newcommand{\n}[2]{N_{#1}(#2)}
\newcommand{\nn}[1]{N(#1)}

\newcommand{\op}[2]{O_{#1}(#2)}

\newcommand{\oupper}[2]{O^{#1}(#2)}
\newcommand{\ostar}[1]{O_{*}(#1)}

\newcommand{\ff}[1]{F(#1)}
\newcommand{\compp}[1]{\operatorname{comp}(#1)}
\newcommand{\comp}[2]{\operatorname{comp}_{#1}(#2)}
\newcommand{\compsol}[1]{\comp{\mathrm{sol}}{#1}}
\newcommand{\compasol}[1]{\comp{A,\mathrm{sol}}{#1}}
\newcommand{\layerr}[1]{E(#1)}
\newcommand{\layer}[2]{E_{#1}(#2)}
\newcommand{\layersol}[1]{\layer{\mathrm{sol}}{#1}}
\newcommand{\gfitt}[1]{F^{*}(#1)}
\newcommand{\sol}[1]{\operatorname{sol}(#1)}
\newcommand{\frat}[1]{\Phi(#1)}
\newcommand{\zz}[1]{Z(#1)}

\newcommand{\card}[1]{|\,#1\,|}
\newcommand{\primes}[1]{\pi(#1)}
\newcommand{\syl}[2]{\operatorname{Syl}_{#1}(#2)}
\newcommand{\aut}[1]{\operatorname{Aut}(#1)}

\newcommand{\gen}[2]{\langle \,#1 \mid #2\, \rangle}
\newcommand{\listgen}[1]{\langle \,#1\, \rangle}
\newcommand{\hyp}[1]{\operatorname{\rm Hyp}(#1)}
\newcommand{\rank}[1]{\operatorname{rank}(#1)}

\newcommand{\set}[2]{\{ \;#1 \mid #2\;\}}
\newcommand{\listset}[1]{\{ \,#1\, \}}

\newcommand{\ltwo}[1]{\operatorname{L}_{2}(#1)}
\newcommand{\uthree}[1]{\operatorname{U}_{3}(#1)}
\newcommand{\sz}[1]{\operatorname{Sz}(#1)}
\newcommand{\sym}[1]{\operatorname{Sym}(#1)}

\newcommand{\badfourlist}{$\ltwo{2^{r}}$, $\ltwo{3^{r}}$, $\uthree{2^{r}}$ or $\sz{2^{r}}$}

\newcommand{\kernel}[1]{\ker #1}

\newcommand{\isomorphic}{\cong}
\newcommand{\normal}{\,\unlhd\,}
\newcommand{\subnormal}{\,\unlhd\unlhd\,}
\newcommand{\characteristic}{\operatorname{char}}

\newcommand{\br}[1]{\overline{#1}}
\newcommand{\nonid}{^\#}

\newcommand{\fancyP}{\mathcal P}
\newcommand{\fancyL}{\mathcal L}
\newcommand{\fancyLstar}{\fancyL^{*}}
\newcommand{\norm}[1]{||#1||}
\newcommand{\thetasol}{\theta_{\mathrm{sol}}}
\newcommand{\thetapsol}{\theta_{p-\mathrm{sol}}}
\newcommand{\linksto}{\rightsquigarrow}

\theoremstyle{plain}

\newtheorem{Theorem}{Theorem}[section]

\newtheorem{Lemma}[Theorem]{Lemma}
\newtheorem{Corollary}[Theorem]{Corollary}
\newtheorem{Hypothesis}[Theorem]{Hypothesis}

\newtheorem*{SFTheorem}{The Signalizer Functor Theorem}

\newtheorem{Claim}{Claim}

\theoremstyle{definition}
\newtheorem{Definition}[Theorem]{Definition}

\theoremstyle{remark}

\newtheorem*{UnnumberedRemark}{Remark}

\numberwithin{equation}{section}

\begin{document}

\title{A new proof of the Nonsolvable Signalizer Functor Theorem}


\author{Paul Flavell}
\address{The School of Mathematics\\University of Birmingham\\Birmingham B15 2TT\\Great Britain}
\email{P.J.Flavell@bham.ac.uk}
\thanks{A considerable portion of this research was done whilst the author was in receipt
of a Leverhulme Research Project grant and during visits to the Mathematisches Seminar,
Christian-Albrechts-Universit\"{a}t, Kiel, Germany.
The author expresses his thanks to the Leverhulme Trust for their support and
to the Mathematisches Seminar for its hospitality.}


\date{}

\begin{abstract}
    The Signalizer Functor Method as developed by Gorenstein and Walter
    played a fundamental role in the first proof of the Classification of
    the Finite Simple Groups.
    It plays a similar role in the new proof  of the Classification in the
    Gorenstein-Lyons-Solomon book series.
    The key results are Glauberman's Solvable Signalizer Functor Theorem
    and McBride's Nonsolvable Signalizer Functor Theorem.
    Given their fundamental role,
    it is desirable to have new and different proofs of them.
    This is accomplished in
    {\em A new proof of the Solvable Signalizer Functor Theorem,}
    P. Flavell, J. Algebra, 398 (2014) 350--363 for Glauberman's Theorem.
    The purpose of this paper is to give a new proof of McBride's Theorem.
\end{abstract}

\maketitle
\section{Introduction}\label{intro}
The Signalizer Functor Method as developed by Gorenstein and Walter
played a fundamental role in the first proof of the Classification
of the Finite Simple Groups.
It plays a similar role in the new proof of the Classification in
the Gorenstein-Lyons-Solomon book series \cite{GLS1}.
A discussion of the method may be found in \cite{ALSS,GLS1,GLS2,Gor}.
The key results being Glauberman's Solvable Signalizer Functor Theorem
\cite{GG} and McBride's Nonsolvable Signalizer Functor Theorem \cite{McB1,McB2}.
They are taken as background results in the Gorenstein-Lyons-Solomon project
and not reproved there.
Given their fundamental role,
it is desirable to have new and different proofs.
This is accomplished in \cite{PFsft} for Glauberman's Theorem.
The purpose of this paper is to give a new proof of McBride's Theorem.

We have taken the liberty of combining the theorems of Glauberman and McBride
into a single result.
We shall prove:

\begin{SFTheorem}
    Let $A$ be a finite abelian group of rank at least $3$ that acts on
    the group $G$.
    Let $\theta$ be an $A$-signalizer functor on $G$ and assume that
    $\theta(a)$ is a $K$-group for all $a \in A\nonid$.
    Then $\theta$ is complete.

    Moreover, the composition factors of the completion of $\theta$
    are to be found amongst the composition factors
    of the subgroups $\theta(a); a \in A\nonid$.
\end{SFTheorem}

Recall that by definition,
$\theta$ ia a mapping that assigns to each $a \in A\nonid$
a finite $A$-invariant subgroup $\theta(a)$ of $\cz{G}{a}$
with order coprime to $\card{A}$ that satisfies \[
    \theta(a) \cap \cz{G}{b} \leq \theta(b)
\]
for all $a,b \in A\nonid$.
Note that $G$ is not assumed to be finite.
To say that $\theta$ is complete means there exists a finite
$A$-invariant subgroup $K$,
of order coprime to $\card{A}$,
such that \[
    \theta(a) = \cz{K}{a}
\]
for all $a \in A\nonid$.
In particular,
the subgroup generated by the subgroups $\theta(a)$ is finite
with order coprime to $\card{A}$.
An exposition of elementary signalizer functor theory
may be found in \cite{PFsft}.

Recall also that a $K$-group is a finite group all of whose simple
sections are known simple groups.
The $K$-group assumption indicates that some portions of the argument
rely on properties of simple groups that are established by taxonomy.
The main application of the Signalizer Functor Theorem is to
construct large subgroups in a minimal counterexample to the
Classification Theorem.
Thus, whilst not ideal,
the $K$-group assumption causes no difficulty.

The proof of McBride's Theorem presented here is very different from the original.
It is based on the author's proof of Glauberman's Theorem and a
general theory of automorphisms of finite groups as developed in \cite{PFI,PFII,PFpp,PFcas}.
We prefer the view that the Signalizer Functor Theorem is not a single
isolated result but rather one of the high points of a well developed theory
of automorphisms of finite groups.
Indeed, although much of the material in \cite{PFI,PFII,PFpp,PFcas} was motivated
by the present work,
it has been developed in much greater depth and generality than is
required for the proof of the Signalizer Functor Theorem.

Sections \S\ref{prel},\ldots,\S\ref{er} consist mainly of statements of the
general theory required and in \S\ref{m}, the proof begins.

The author would like to thank Professor George Glauberman for his careful
reading of an earlier version of this manuscript.

\section{Preliminaries}\label{prel}
The reader is assumed to be familiar with elementary
signalizer functor theory,
see for example \cite{PFsft} or \cite{KS}.
An understanding of the author's proof of the
Solvable Signalizer Functor Theorem \cite{PFsft}
would be advantageous.

Unless stated otherwise,
the word group will mean finite group.
The reader is assumed to be familiar with the notions of
the Fitting subgroup, the set of components, the layer
and the generalized Fitting subgroup of a group $G$ denoted by
$\ff{G}, \compp{G}, \layerr{G}$ and $\gfitt{G}$ respectively.
See for example \cite{KS}.
The notation $\sol{G}$ is used to denote the largest normal
solvable subgroup of $G$.
We will need a number of variations of the notion of component
as developed in  \cite{PFI}.

\begin{Definition} \label{p:1}
    A \emph{sol-component} of $G$ is a perfect subnormal subgroup
    of $G$ that maps onto a component of $G/\sol{G}$.
    The set of sol-components of $G$ is denoted by \[
        \compsol{G}
    \]
    and we define \[
        \layersol{G} = \listgen{\compsol{G}} \quad\text{and}\quad
        \ostar{G} = \sol{G}\layersol{G}.
    \]
\end{Definition}

\begin{Lemma} \label{p:2}
    Let $G$ be a group.
    \begin{enumerate}
        \item[(a)]  The sol-components of $G$ are the minimal nonsolvable
                    subnormal subgroups of $G$.

        \item[(b)]  Set $\br{G} = G/\sol{G}$.
                    The map $K \mapsto \br{K}$ is a bijection
                    $\compsol{G} \longrightarrow \compp{G}$.

        \item[(c)]  If $K \in \compsol{G}$ and $N \subnormal G$
                    then $K \leq N$ or $N \leq \n{G}{K}$.

        \item[(d)]  Distinct sol-components of $G$ normalize each other
                    and commute modulo $\sol{G}$.

        \item[(e)]  If $K \in \compsol{G}$ then $K \normal \ostar{G}$.

        \item[(f)]  (McBride) If $H$ satisfies $\ostar{G} \leq H \leq G$
                    then $\ostar{G} = \ostar{H}$.
    \end{enumerate}
\end{Lemma}
\begin{proof}
    (a),\ldots,(e) are well known and elementary,
    see for example \cite[Lemma~3.2]{PFI}.
    For (f) see \cite[Lemma~2.15]{McB1} or \cite[Lemma~8.2]{PFI}.
\end{proof}

Next we bring in a group of automorphisms.

\begin{Definition} \label{p:3}
    Let the group $A$ act on the group $G$.
    \begin{enumerate}
        \item[(a)]  $G$ is \emph{$A$-simple} if $G$ is nonabelian and
                    the only $A$-invariant normal subgroups of $G$
                    are $1$ and $G$.

        \item[(b)]  $G$ is \emph{$A$-quasisimple} of $G$ is perfect
                    and $G/\zz{G}$ is $A$-simple.

        \item[(c)]  An \emph{$A$-component} of $G$ is the subgroup
                    generated by an orbit of $A$ on $\compp{G}$.
                    The set of $A$-components of $G$ is denoted by \[
                        \comp{A}{G}.
                    \]

        \item[(d)]  An \emph{$(A, \mbox{sol})$-component} of $G$ is the subgroup
                    generated by an orbit of $A$ on $\compsol{G}$.
                    The set of $(A, \mbox{sol})$-components of $G$ is denoted by \[
                        \compasol{G}.
                    \]
    \end{enumerate}
\end{Definition}

\noindent The $A$-components of $G$ are the subnormal $A$-quasisimple subgroups of $G$.
The $(A,\mbox{sol})$-components of $G$ are the minimal nonsolvable $A$-invariant
subnormal subgroups of $G$.
A result exactly analogous to Lemma~\ref{p:2} holds for $(A, \mbox{sol})$-components.

Recall that a group $X$ is \emph{semisimple} if $X = \layerr{X}$
and \emph{constrained} if $\cz{X}{\ff{X}} \leq \ff{X}$.
Then any $(A, \mbox{sol})$-component of $G$ is either semisimple or constrained.

The group $A$ acts \emph{coprimely} on the group $G$ if $A$ acts on $G$;
the orders of $A$ and $G$ are coprime; and $A$ or $G$ is solvable.
If $p$ is a prime then we denote by \[
    \syl{p}{G;A}
\]
the set of maximal $A$-invariant $p$-subgroups of $G$
with respect to inclusion.

\begin{Theorem}[Coprime Action] \label{p:4}
    Suppose the group $A$ acts coprimely on the group $G$.
    \begin{enumerate}
        \item[(a)]  Let $p$ be a prime.
                    Then $\syl{p}{G;A} \subseteq \syl{p}{G}$
                    and $\cz{G}{A}$ acts transitively by conjugation
                    on $\syl{p}{G;A}$.

        \item[(b)]  Let $N$ be an $A$-invariant normal subgroup of $G$
                    and set $\br{G} = G/N$.
                    Then $\cz{\br{G}}{A} = \br{\cz{G}{A}}$.

        \item[(c)]  $G = [G,A]\cz{G}{A}$ and $[G,A] = [G,A,A]$.

        \item[(d)]  Suppose $A$ is elementary and noncyclic.
                    Then \[
                        G = \gen{\cz{G}{B}}{B \in \hyp{A}}%
                          = \gen{\cz{G}{a}}{a \in A\nonid}.
                    \]
                    Moreover if $T \leq A$ then \[
                        [G,T] = \gen{ [\cz{G}{B},T] }{B \in \hyp{A}} %
                              = \gen{ [\cz{G}{a},T] }{a \in A\nonid}.
                    \]

        \item[(e)]  If $G = XY$ where $X$ and $Y$ are $A$-invariant
                    subgroups of $G$ then $\cz{G}{A} = \cz{X}{A}\cz{Y}{A}$.

        \item[(f)]  Suppose $K \subnormal G$ and $[K,A] = [\cz{G}{K},A] = 1$.
                    Then $[G,A] = 1$.

        \item[(g)]  If $[\gfitt{G},A] = 1$ then $[G,A] = 1$.

        \item[(h)]  Suppose that $G$ is $p$-solvable for some prime $p$
                    and that $A$ centralizes a Sylow $p$-subgroup of $G$.
                    Then $[G,A] \leq \op{p'}{G}$.

        \item[(i)]  Suppose that $G$ is a $p$-group for some prime $p$;
                    that $A$ centralizes every characteristic abelian subgroup of $G$
                    and that $G = [G,A]$.
                    Then \[
                        G' = \frat{G} = \zz{G} = \cz{G}{A}.
                    \]
    \end{enumerate}
\end{Theorem}
\begin{proof}
    For (a),(b),(c),(e) see \cite[p.184--187]{KS}.
    (d) is \cite[p.484]{SuzII}.

    (f). By induction, we may suppose $K \normal G$.
    Then $[G,A,A] \leq [\cz{G}{K},A] = 1$.
    Apply (c).

    (g). Since $\cz{G}{\gfitt{G}} = \zz{\ff{G}}$,
    this follows from (f).

    (h). Set $\br{G} = G/\op{p'}{G}$ so $\gfitt{\br{G}} = \op{p}{\br{G}}$.
    Then $[\gfitt{\br{G}},A] = 1$. Apply (g).

    (i). This is well known, see \cite[Corollary~3.3]{PFsol} for example.
\end{proof}

\noindent Note that (a) implies that for each prime $p$,
$G$ possesses a unique maximal $A\cz{G}{A}$-invariant $p$-subgroup,
namely the intersection of the members of $\syl{p}{G;A}$.

\begin{Definition} \label{p:5}
    Suppose that group $A$ acts coprimely on the group $G$.
    Let $p$ be a prime.
    Then \[
        \op{p}{G;A}
    \]
    is the intersection of all the $A$-invariant $p$-subgroups of $G$.
\end{Definition}

Finally, we collect together some more specialized results.

\begin{Lemma}\label{p:6}
    Let $R$ be an elementary abelian $r$-group that acts
    coprimely on the $K$-group $X$.
    \begin{enumerate}
        \item[(a)]   $\op{p}{X;R}' \leq \sol{X}$ for all primes $p$.

        \item[(b)]  Suppose $R$ is noncyclic.
                    Then \[
                        \bigcap_{b \in R\nonid}\sol{\cz{X}{b}} \leq \sol{X}.
                    \]

        \item[(c)]  Suppose $R$ is cyclic, $X = [X,R]$, $t$ is a prime and
                    $RX$ acts on the $t$-group $T$ with $\cz{T}{R} = 1$.
                    Set $\br{X} = X/\cz{X}{T}$.
                    Then $\primes{\br{X}} \subseteq \listset{2,t}$
                    and $\br{X}/\op{t}{\br{X}}$ is either trivial or a
                    nonabelian $2$-group.
    \end{enumerate}
\end{Lemma}
\begin{proof}
    (a). This is \cite[Theorem~3.1(c)]{PFpp}.

    (b). Because if $H$ is a simple $K$-group with order coprime to $r$
    then the Sylow $r$-subgroups of $\aut{K}$ are cyclic.

    (c). This reduces to the case where $T$ is elementary abelian and
    $RX$ acts nontrivially and irreducibly on $T$.
    \cite[Theorem~7.1]{PFI} implies $\br{X}$ is a special $2$-group.
\end{proof}

\section{$A$-simple groups} \label{as}
In the proof of the Signalizer Functor Theorem presented here,
much of the argument concerns $A$-components.
Consequently it is necessary to have an understanding
of $A$-simple groups.
Throughout this section, \[
    \mbox{$r$ is a prime and $A \not= 1$ is an elementary abelian $r$-group.}
\]

\begin{Theorem}\label{as:1}
    Suppose that $A$ acts faithfully and coprimely on the $K$-group $K$
    and that $K$ is $A$-simple.
    \begin{enumerate}
        \item[(a)]  $K = K_{1} \times\cdots\times K_{n}$ where $\listset{K_{1},\ldots,K_{n}}$
                    is a collection of simple subgroups of $K$ that is permuted
                    transitively by $A$.
    \end{enumerate}
    Define \[
        A_{\infty} = \kernel{A \longrightarrow \sym{\listset{K_{1},\ldots,K_{n}}}}.
    \]
    \begin{enumerate}
        \item[(b)]  $\card{A_{\infty}} = 1$ or $r$.

        \item[(c)]  Let $a \in A \setminus A_{\infty}$.
                    Then $\cz{K}{a}$ is a maximal $A\cz{K}{A}$-invariant
                    proper subgroup of $K$.
                    It is $A$-simple and has $\card{A}/r\card{A_{\infty}}$
                    components,
                    each of which is normalized by $A_{\infty}$.
                    Moreover $\cz{A}{\cz{K}{a}} = \listgen{a}$.

        \item[(d)]  Let $a \in A_{\infty}\nonid$.
                    Then either $\cz{K}{a}$ is solvable or
                    $\gfitt{\cz{K}{a}}$ is $A$-simple.
                    In the latter case,
                    $\cz{K}{a}/\gfitt{\cz{K}{a}}$ is abelian.

        \item[(e)]  Assume that $\cz{K}{A}$ is solvable.
                    \begin{enumerate}
                        \item[(i)]  $\card{A_{\infty}} = r$ and $K_{1}$ is isomorphic
                                    to \badfourlist.

                        \item[(ii)] $K$ possesses a unique maximal $A\cz{K}{A}$-invariant
                                    solvable subgroup $S$.

                        \item[(iii)] $\cz{K}{A_{\infty}} \leq S$ and $S$ is maximal subject
                                    to being an $A\cz{K}{A}$-invariant proper subgroup of $K$.
                    \end{enumerate}

        \item[(f)]  Assume that $\cz{K}{A}$ is nonsolvable.
                    \begin{enumerate}
                        \item[(i)]  $\gfitt{\cz{K}{A}}$ is simple and
                                    $\cz{K}{A}/\gfitt{\cz{K}{A}}$ is cyclic.

                        \item[(ii)] $K$ does not possess a nontrivial $A\cz{K}{A}$-invariant
                                    solvable subgroup.
                    \end{enumerate}
    \end{enumerate}
\end{Theorem}
\begin{proof}
    See \cite[\S 6]{PFI}.
\end{proof}

We note in particular that if $a \in A\nonid$ then either \[
    \mbox{$\gfitt{\cz{K}{a}}$ is $A$-simple or $\cz{K}{a}$ is solvable.}
\]
By (d), the following \emph{balance} property holds,
for all $a,b \in A\nonid$ \[
    \layerr{\layerr{\cz{K}{a}} \cap \cz{K}{b}} \leq \layerr{\cz{K}{b}}.
\]
These properties characterize $K$ and the collection $\set{\cz{K}{a}}{a \in A\nonid}$
of fixed point subgroups.
It is convenient to state this characterization in the language of
signalizer functor theory.

\begin{Theorem}[Characterization of $A$-Simple Groups \cite{PFcas}] \label{as:2}
    Suppose that $\rank{A} \geq 3$ and that $A$ acts on the (possibly infinite) group $G$.
    Assume the following:
    \begin{enumerate}
        \item[(i)]  $\theta$ is an $A$-signalizer functor on $G$.

        \item[(ii)] $\theta(a)$ is a $K$-group for all $a \in A\nonid$.

        \item[(iii)]   If $a \in A\nonid$ with $\layerr{\theta(a)} \not= 1$ then
                    $\layerr{\theta(a)}$ is $A$-simple, $\ff{\theta(a)} = 1$
                    and $\cz{A}{\layerr{\theta(a)}} = \listgen{a}$.

        \item[(iv)] For all $a,b \in A\nonid$, \[
                        \layerr{\layerr{\theta(a)} \cap \cz{G}{b}} \leq \layerr{\theta(b)}.
                    \]

        \item[(v)]  $G = \gen{ \layerr{\theta(a)}}{a \in A\nonid} \not= 1$.
    \end{enumerate}
    Then $G$ is a finite $r'$-group,
    it is $A$-simple,
    a $K$-group and \[
        \theta(a) = \cz{G}{a}
    \]
    for all $a \in A\nonid$.
    In particular $\theta$ is complete and $G$ is its completion.
\end{Theorem}

We close this section with three results on $A$-quasisimple groups.

\begin{Definition} \label{as:3}
    Whenever $K$ is an $A$-quasisimple group define \[
        \cstar{K}{A} = \left\{\begin{array}{c@{$\;\;$}l}
                                \cz{K}{A} & \mbox{if $\cz{K}{A}$ is solvable} \\
                                \layerr{\cz{K}{A}} & \mbox{if $\cz{K}{A}$ is nonsolvable.}
                        \end{array}\right.
    \]
\end{Definition}

\begin{Lemma}\label{as:4}
    Suppose that $K$ is an $A$-quasisimple $K$-group on
    which $A$ acts coprimely.
    \begin{enumerate}
        \item[(a)]  $\cstar{K}{A}$ is nonabelian.

        \item[(b)]  Suppose that $H$ is an $A\cz{K}{A}$-invariant
                    nonsolvable subgroup of $K$.
                    Then \[
                        \mbox{$H^{(\infty)} = \layerr{H}$ is $A$-quasisimple %
                        and $\cstar{\layerr{H}}{A} = \cstar{K}{A}$.}
                    \]
    \end{enumerate}
\end{Lemma}
\begin{proof}
    Set $\br{K} = K/\zz{K}$,
    so $\br{K}$ is $A$-simple.
    Coprime Action(b) implies $\cz{\br{K}}{A} = \br{\cz{K}{A}}$.
    Write $\br{K} = \br{K}_{1} \times\cdots\times \br{K}_{n}$ where
    $\listset{\br{K}_{1},\ldots,\br{K}_{n}}$ is a collection of
    simple subgroups of $\br{K}$ that is permuted transitively by $A$.
    Set $A_{\infty} = \ker A \longrightarrow \sym{\listset{\br{K}_{1},\ldots,\br{K}_{n}}}$.
    Recall that if $X$ is a group and $Z \leq \zz{X}$ then
    $\layerr{X}$ maps onto $\layerr{X/Z}$.

    (a). If $A_{\infty} = \cz{A}{\br{K}}$ then \cite[Lemma~6.5]{PFI}
    implies $\cz{\br{K}}{A} \isomorphic \br{K}_{1}$,
    so $\cz{\br{K}}{A}$ is simple.
    If $A_{\infty} \not\leq \cz{A}{\br{K}}$ then \cite[Lemma~6.5]{PFI}
    implies $\cz{\br{K}}{A} \isomorphic \cz{\br{K}_{1}}{A_{\infty}}$
    and then \cite[Theorem~4.1]{PFI} implies $\gfitt{\cz{\br{K}_{1}}{A_{\infty}}}$
    is simple or $\cz{\br{K}_{1}}{A_{\infty}}$ is solvable and nonabelian.
    Since $\layerr{\cz{K}{A}}$ maps onto $\layerr{\cz{\br{K}}{A}}$
    it follows that $\cstar{K}{A}$ is nonabelian.

    (b). Recall from \cite{PFI} that $H$ is overdiagonal if $H$ projects
    onto each $\br{K}_{i}$.
    In the contrary case,
    $H$ is underdiagonal.
    Suppose that $H$ is overdiagonal.
    \cite[Lemma~6.6]{PFI} implies $H = \cz{K}{B}(H \cap \zz{K})$
    for some $B \leq A$ with $B \cap A_{\infty} = \cz{A}{K}$.
    Then \cite[Lemma~6.5]{PFI} implies $\cz{K}{B}$ is $A$-quasisimple.
    Consequently $H^{(\infty)} = \layerr{H} = \cz{K}{B}$
    and as $B \leq A$ we have $\cz{\layerr{H}}{A} = \cz{K}{A}$
    and the conclusion holds in this case.
    Hence we may assume that $H$ is underdiagonal.

    If $\br{K}$ possesses a nontrivial $A\cz{K}{A}$-invariant solvable subgroup
    then all $A\cz{\br{K}}{A}$-invariant
    underdiagonal subgroups are solvable by \cite[Lemma~6.7]{PFI}.
    Thus $\br{K}$ possesses no such subgroup.
    In particular,
    $\cz{\br{K}}{A}$ is nonsolvable,
    whence $\gfitt{\cz{\br{K}}{A}}$ is simple.
    Also, $\ff{\br{H}} = 1$ so we may choose $\br{H}_{0} \in \comp{A}{\br{H}}$.

    Since $\br{H}_{0}$ is nonsolvable,
    \cite[Theorem~4.4]{PFI} implies $\cz{\br{H}_{0}}{A} \not= 1$.
    Now $\cz{\br{H}_{0}}{A} \subnormal \cz{\br{K}}{A}$ whence
    $\gfitt{\cz{\br{H}_{0}}{A}} = \gfitt{\cz{\br{K}}{A}}$
    and $\br{H}_{0}$ is uniquely determined.
    Then $\layerr{\br{H}} = \br{H}_{0}$ and $\layerr{\br{H}}$ is $A$-simple.
    Now $\gfitt{\cz{\br{K}}{A}} \leq \layerr{\br{H}}$.
    Recall that $\cz{\br{K}}{A}/\gfitt{\cz{\br{K}}{A}}$ is cyclic.
    Consequently $\cz{\br{H}/\layerr{\br{H}}}{A}$ is cyclic so
    \cite[Theorem~4.4]{PFI} implies $\br{H}/\layerr{\br{H}}$ is solvable.
    We have shown that \[
        \mbox{$\br{H}^{(\infty)} = \layerr{\br{H}}$ is $A$-simple and %
            $\layerr{\cz{\layerr{\br{H}}}{A}} = \layerr{\cz{\br{K}}{A}} \not= 1$.}
    \]
    Then $H^{(\infty)}\zz{K} = \layerr{H}\zz{K}$ so taking the derived group
    yields $H^{(\infty)} = \layerr{H}$.
    Similarly $\layerr{\cz{\layerr{H}}{A}} = \layerr{\cz{K}{A}}$,
    completing the proof.
\end{proof}

\begin{Lemma}[{\cite[Lemma~6.12]{PFI}}]\label{as:5}
    Suppose $A$ acts coprimely on the $A$-quasisimple group $K$.
    \begin{enumerate}
        \item[(a)]  If $A$ is noncyclic then \[
            K = \gen{ \cz{K}{D} }{ \mbox{$D \in \hyp{A}$ and $\cz{K}{d}$ is $A$-quasisimple for all $d \in D\nonid$}}.
                \]

        \item[(b)]  If $D \in \hyp{A}$ and $D$ is noncyclic then \[
            K = \gen{ \cz{K}{d} }{ \mbox{$d \in D\nonid$ and $\cz{K}{d}$ is $A$-quasisimple} }.
            \]
    \end{enumerate}
\end{Lemma}

\begin{Lemma}[{\cite[Theorem~4.4(c)]{PFI}}]\label{as:6}
    Suppose $A$ acts coprimely on the $K$-group $G$ and
    that $K \in \comp{A}{G}$.
    Then $\cz{G}{\cz{K}{A}} = \cz{G}{K}$.
\end{Lemma}

\section{Automorphisms}\label{aut}
Throughout this section we assume:
\begin{Hypothesis} \label{aut:1} \mbox{}
    \begin{itemize}
        \item   $r$ is a prime and $A \not= 1$ is an elementary abelian $r$-group.

        \item   $A$ acts coprimely on the $K$-group $G$.

        \item   $a \in A\nonid$.

        \item   $H$ is an $A\cz{G}{a}$-invariant subgroup of $G$.
    \end{itemize}
\end{Hypothesis}

\noindent The following result relates the structure of $H$
to the structure of $G$ in the case that $G$ is solvable.

\begin{Theorem} \label{aut:2}
    Assume that $G$ is solvable and $H = [H,a]$.
    \begin{enumerate}
        \item[(a)]  Let $p$ be a prime.
                    Then \[
                        \op{p}{H} \leq \op{p}{G}
                    \]
                    or all of the following hold:
                    $p=2$, $r$ is Fermat and
                    the Sylow $2$-subgroups of $H$ are nonabelian.

        \item[(b)]  $\op{2}{\oupper{2}{H}} \leq \op{2}{G}$.
    \end{enumerate}
\end{Theorem}

\noindent This result is fundamental to the author's proof of the
Solvable Signalizer Functor Theorem.
It is a consequence of well known results on the
representation theory of solvable groups.
See \cite[Corollary~5.2]{PFsol} for example.
To deal with nonsolvable signalizer functors,
it is necessary to have analogous results for nonsolvable groups.

\begin{Theorem} \label{aut:3}\mbox{}
    \begin{enumerate}
        \item[(a)]  Suppose $K \in \compasol{H}$ and $K = [K,a]$.
                    Then $K \in \compasol{G}$.

        \item[(b)]  $[\ostar{H},a]^{(\infty)} \normal \ostar{G}$.
    \end{enumerate}
\end{Theorem}
\begin{proof}
    (a). This is \cite[Theorem~7.5(a)]{PFII}.

    (b). Note that if $K \in\compasol{H}$ then either $[K,a] \leq \sol{K} \leq \sol{H}$
    or $K = [K,a]$.
    Then $[\ostar{H},a]^{(\infty)} = \gen{K \in \compasol{H}}{K = [K,a]}$.
    Apply (a).
\end{proof}

\begin{Theorem}\label{aut:4}
    Suppose $K \in \comp{A}{H}$.
    Then there exists $\widetilde{K}$ with \[
        K \leq \widetilde{K} \in \compasol{G}.
    \]
    Moreover:
    \begin{enumerate}
        \item[(a)]  If $[K,a] \not= 1$ then $K = [K,a] = \widetilde{K}$.

        \item[(b)]  If $[K,a] = 1$ then $K = \layerr{\cz{\widetilde{K}}{a}}$.

        \item[(c)]  Suppose $\widetilde{K}$ is constrained.
                    Then $[K,a] = 1$ and \[
                        \widetilde{K} = K\sol{\widetilde{K}}.
                    \]
                    Moreover if $b \in A \setminus \cz{A}{K}$ then
                    $\widetilde{K} = \listgen{K, \cz{\sol{\widetilde{K}}}{b}}$.

        \item[(d)]  Let $L \in \compasol{G}$.
                    Assume \[
                        \mbox{$L \not= \widetilde{K}$ and $L = [L,a]$.}
                    \]
                    Then $[\widetilde{K},L] = 1$.
    \end{enumerate}
\end{Theorem}
\begin{proof}
    This is \cite[Theorem~9.8]{PFI} except for the final assertion in (c)
    which is \cite[Lemma~8.2]{PFI}.
\end{proof}

\begin{UnnumberedRemark}
       In the constrained case, it is in fact possible to show that
       $\widetilde{K} = K\ff{\widetilde{K}}$,
       but we do not need this stronger result.
       Recall that distinct $A$-components of $G$ commute.
       This fact is very useful.
       However, the same is not necessarily true of $(A, \mbox{sol})$-components.
       (d) circumvents this difficulty.
\end{UnnumberedRemark}

\section{$\fancyP$-subgroups} \label{psgps}
Throughout this section, we assume:

\begin{Hypothesis} \label{psgps:1} \mbox{}
    \begin{itemize}
        \item   $r$ is a prime and $A \not= 1$ is an elementary abelian $r$-group.

        \item   $\fancyP$ is a group theoretic property that is closed under
                subgroups, quotients and extensions.
    \end{itemize}
\end{Hypothesis}

\begin{Definition} \label{psgps:2}
    Suppose $A$ acts on the group $G$.
    \begin{align*}
        \op{\fancyP}{G} &= %
                    \gen{ X }{ \mbox{$X$ is an $A$-invariant normal $\fancyP$-subgroup of $G$} }.\\
        \op{\fancyP}{G;A} &= %
                    \gen{ X }{ \mbox{$X$ is an $A\cz{G}{A}$-invariant $\fancyP$-subgroup of $G$} }.
    \end{align*}
\end{Definition}

\noindent It is clear that $\op{\fancyP}{G}$ is itself a $\fancyP$-group
and is thus the unique maximal normal $\fancyP$-subgroup of $G$.
The following is less clear:

\begin{Theorem}[{\cite[Theorem~5.2]{PFII}}] \label{psgps:3}
    Suppose $A$ acts coprimely on the $K$-group $G$.
    Then $\op{\fancyP}{G;A}$ is a $\fancyP$-group.
    In other words,
    $G$ possesses a unique maximal $A\cz{G}{A}$-invariant $\fancyP$-subgroup.
\end{Theorem}

A useful corollary is the following:

\begin{Corollary}\label{psgps:4}
    Assume the hypotheses of Theorem~\ref{psgps:3}.
    Suppose that $N$ is an $A$-invariant subnormal subgroup of $G$.
    Then \[
        \op{\fancyP}{N;A} = \op{\fancyP}{G;A} \cap N.
    \]
\end{Corollary}
\begin{proof}
    Consider first the case that $N \normal G$.
    Then $\op{\fancyP}{G;A} \cap N$ is an $A\cz{N}{A}$-invariant
    $\fancyP$-subgroup of $N$ so $\op{\fancyP}{G;A} \cap N \leq \op{\fancyP}{N;A}$.
    Now $\cz{N}{A} \normal \cz{G}{A}$ so it follows that $\cz{G}{A}$ permutes the
    $A\cz{N}{A}$-invariant $\fancyP$-subgroups of $N$.
    Then $\cz{G}{A}$ normalizes $\op{\fancyP}{N;A}$.
    Theorem~\ref{psgps:3} implies $\op{\fancyP}{N;A}$ is a $\fancyP$-subgroup
    so $\op{\fancyP}{N;A} \leq \op{\fancyP}{G;A}$
    and the result follows in this case.

    Suppose that $N$ is not normal in $G$.
    Set $G_{0} = \listgen{N^{G}}$.
    Since $N$ is a proper subnormal subgroup of $G$ it follows that
    $G_{0}$ is a proper $A$-invariant normal subgroup of $G$.
    Apply the previous case and induction.
\end{proof}

The main result of this section is the following:
\begin{Theorem} \label{psgps:5}
    Suppose $A$ acts on the (possibly infinite) group $G$
    and that $\theta$ is an $A$-signalizer functor on $G$.
    Assume that $\theta(a)$ is a $K$-group for all $a \in A\nonid$.
    Define $\theta_{\fancyP}$ by \[
        \theta_{\fancyP}(a) = \op{\fancyP}{\theta(a);A}.
    \]
    \begin{enumerate}
        \item[(a)]  $\theta_{\fancyP}$ is an $A$-signalizer functor on $G$.

        \item[(b)]  Assume that $A$ is noncyclic;
                    that $\theta_{\fancyP}$ is complete;
                    and that $\theta_{\fancyP}(G)$ is a $K$-group.
                    Then $\theta_{\fancyP}(G)$ is the unique maximal $\theta(A)$-invariant
                    $(\fancyP, \theta)$-subgroup of $G$.
                    (A $(\fancyP, \theta)$-subgroup is a $\theta$-subgroup that is a $\fancyP$-group.)
    \end{enumerate}
\end{Theorem}

\noindent We also need the following:

\begin{Lemma}\label{psgps:6}
    Suppose that $A$ acts coprimely on the $K$-group $G$.
    Assume that $A$ is noncyclic and that $\cz{G}{a}$ is
    a $\fancyP$-group for all $a \in A\nonid$.
    Then $G$ is a $\fancyP$-group.
\end{Lemma}
\begin{proof}
    Using Coprime Action(b) we may suppose that $1$ and $G$ are the
    only $A$-invariant normal subgroups of $G$.
    Then $G$ is characteristically simple.
    Suppose $G$ is abelian.
    Then $G$ is an elementary abelian $p$-group for some prime $p$.
    Coprime Action(d) implies $\cz{G}{a} \not= 1$ for some $a \in A\nonid$.
    Since $\cz{G}{a}$ is $A$-invariant and normal we have $G = \cz{G}{a}$
    so $G$ is a $\fancyP$-group.
    Hence we may suppose that $G$ is nonabelian.
    Then $G = G_{1} \times\cdots\times G_{n}$ where $\listset{G_{1},\ldots,G_{n}}$
    is a collection of simple subgroups of $G$ that is permuted
    transitively by $A$.
    Suppose that $n > 1$.
    Choose $a \in A$ such that $G_{1}^{a} \not= G_{1}$.
    Then $\set{ gg^{a} \cdots g^{a^{r-1}}}{ g \in G_{1} }$
    is a normal subgroup of $\cz{G}{a}$ that is isomorphic to $G_{1}$.
    Then $G_{1}$ is a $\fancyP$-group,
    whence $G$ is also.
    Suppose that $n = 1$.
    Then $G$ is a simple $K$-group.
    Consequently the Sylow $r$-subgroups of $\aut{G}$ are cyclic
    so $G = \cz{G}{a}$ for some $a \in A\nonid$ and $G$ is a $\fancyP$-group.
\end{proof}

\begin{proof}[Proof of Theorem~\ref{psgps:5}]
    (a). Let $a,b \in A\nonid$.
    Note that $\cz{\theta(a)}{A} = \theta(A) = \cz{\theta(b)}{A}$.
    Now \[
        \theta_{\fancyP}(a) \cap \cz{G}{b} \leq %
        \theta_{\fancyP}(a) \cap \theta(b) \leq \theta_{\fancyP}(b),
    \]
    the first inclusion because $\theta$ is an $A$-signalizer functor
    and the second because $\theta_{\fancyP}(a) \cap \theta(b)$
    is an $A\theta(A)$-invariant $\fancyP$-subgroup of $\theta(b)$.
    Hence $\theta_{\fancyP}$ is an $A$-signalizer functor.

    (b). Set $K = \theta_{\fancyP}(G)$.
    Since $\theta_{\fancyP}$ is complete,
    $\theta_{\fancyP}(a) = \cz{K}{a}$ for all $a \in A\nonid$.
    Lemma~\ref{psgps:6} implies $K$ is a $\fancyP$-group.
    Suppose $L$ is a $\theta(A)$-invariant $(\fancyP, \theta)$-subgroup of $G$.
    If $a \in A\nonid$ then $\cz{L}{a}$ is a $\theta(A)$-invariant
    $(\fancyP, \theta)$-subgroup of $\theta(a)$,
    whence $\cz{L}{a} \leq \theta_{\fancyP}(a) \leq K$.
    Coprime Action(d) implies $L \leq K$.
\end{proof}

\section{Bender's Maximal Subgroup Theorem}\label{ben}
The aim of this section is to prove  slight extension of a result of Bender \cite[1.7]{Bab}.
Bender's result gave a criterion for two maximal subgroups $M$ and $N$ of
a simple group to be equal.
First we need some definitions.

\begin{Definition}\label{ben:1}
    Suppose that $M$ and $N$ are finite subgroup of a (possibly infinite) group.
    \begin{itemize}
        \item   \emph{$M$ is maximal with respect to $N$} if \[
                    \n{N}{T} \leq M
                \]
                whenever $1 \not= T \characteristic M$
                with $T \leq M \cap N$.

        \item   \emph{$M$ and $N$ are comaximal} if $M$ is maximal with respect to $N$
                and $N$ is maximal with respect to $M$.

        \item   $M \linksto N$ means
                    \[
                        X\cz{\gfitt{M}}{X} \leq N \quad\text{for some}\quad X \subnormal \gfitt{M}.
                    \]

        \item   If $p$ is a prime then \emph{$M$ has characteristic $p$} if $\gfitt{M} = \op{p}{M}$.
    \end{itemize}
\end{Definition}

\begin{Theorem}[Bender's Maximal Subgroup Theorem]\label{ben:2}
    Suppose that $M$ and $N$ are finite subgroups of a (possibly infinite) group, that $M$ is maximal
    with respect to $N$ and that $M \linksto N$.
    \begin{enumerate}
        \item[(a)]  $\layerr{M} \leq N$ and $M \cap \op{p}{N} = 1$ for all $p \not\in \primes{\ff{M}}$.

        \item[(b)]  Assume that $\layerr{M} \not= 1$ or $\card{\primes{\ff{M}}} \geq 2$.
                    Then $\op{p}{N} \leq M$ for all $p \in \primes{\ff{M}}$.

        \item[(c)]  Assume in addition that $N$ is maximal with respect to $M$
                    and that
                    \begin{enumerate}
                        \item[(i)]  $N \linksto M$ or
                        \item[(ii)] $\layerr{N} \leq M$ and $\primes{\ff{N}} \subseteq \primes{\ff{M}}$.
                    \end{enumerate}
                    Then $M = N$ or $M$ and $N$ have characteristic $p$ for some prime $p$.
    \end{enumerate}
\end{Theorem}
\begin{proof}
    This is proved in \cite[1.7]{Bab} under the assumption that $M$ and $N$
    are maximal subgroups of a simple group.
    However, only the stated hypotheses are required.
\end{proof}

The result stated below is used to handle the characteristic $p$ case.
Under the given hypotheses,
it leads to the same conclusion.

\begin{Theorem}[{\cite[Theorem~A]{PFpp}}] \label{ben:3}
    Let $p$ be a prime and suppose $M_{1}$ and $M_{2}$
    are finite subgroups of a (possibly infinite) group with the following properties:
    \begin{itemize}
        \item   $M_{1}$ and $M_{2}$ are comaximal.

        \item   $M_{1}$ and $M_{2}$ are $K$-groups with characteristic $p$.

        \item   For each $i$ there is an elementary abelian group $A_{i}$ that
                acts coprimely on $M_{i}$ and $\op{p}{M_{1};A_{1}} = \op{p}{M_{2};A_{2}}$.
    \end{itemize}
    Then $M_{1} = M_{2}$.
\end{Theorem}

\noindent Unfortunately, at one point in the argument this result is not strong enough.
However, the following result,
provides the necessary extra leverage.
Note that Theorem~\ref{ben:3} is a trivial corollary.

\begin{Theorem}[{\cite[Theorem~4.3]{PFpp}}] \label{ben:4}
    Let $p$ be a prime and suppose that $M$ and $S$ are subgroups of a group.
    Assume that:
    \begin{itemize}
        \item   $M$ and $S$ are finite $K$-groups with characteristic $p$.

        \item   $M$ is maximal with respect to $S$.

        \item   There exist elementary abelian groups $A_{m}$ and $A_{s}$
                that act coprimely on $M$ and $S$ respectively and
                $\op{p}{M;A_{m}} = \op{p}{M;A_{s}}$.
    \end{itemize}
    Set \[
        P = \op{p}{M}\op{p}{S}.
    \]
    Then the following hold:
    \begin{enumerate}
        \item[(a)]  If $\op{p}{M}$ is abelian then $J(P) = J(\op{p}{M})$.

        \item[(b)]  $J(P) = J(\op{p}{S})$.
    \end{enumerate}
\end{Theorem}
\noindent Note that $\op{p}{M} \leq \op{p}{M;A_{m}} = \op{p}{S;A_{s}} \leq S$
whence $P$ is a $p$-group.

\section{Elementary results}\label{er}
A number of elementary results are presented.
In particular, to any signalizer functor $\theta$ we associate
a positive integer $\norm{\theta}$.
Note that the group $G$ in the statement of the Signalizer Functor Theorem
is not assumed to be finite.
Hence this device is needed to enable inductive arguments.

Throughout this section we assume the following:

\begin{Hypothesis}\label{er:1}\mbox{}
    \begin{itemize}
        \item   $A$ is an noncyclic abelian group that acts
                on the (possibly infinite) group $G$.

        \item   $\theta$ is an $A$-signalizer functor on $G$.
    \end{itemize}
\end{Hypothesis}

\begin{Lemma}\label{er:2}
    Let $B \leq A$ be noncyclic and define a $B$-signalizer functor
    $\theta_{0}$ by $\theta_{0}(b) = \theta(b)$ for all $b \in B\nonid$.
    If $\theta_{0}$ is complete then so is $\theta$ and $\theta(G) = \theta_{0}(G)$.
\end{Lemma}
\begin{proof}
    Let $K = \theta_{0}(G)$,
    so $K$ is $A$-invariant and $\theta(b) = \theta_{0}(b) = \cz{K}{b}$
    for all $b \in B\nonid$.
    Let $a \in A\nonid$.
    Note that $\theta(a) \cap \cz{G}{b} = \cz{G}{a} \cap \theta(b)$ for all $b \in B\nonid$.
    Using Coprime Action(d) we have
    \begin{align*}
        \theta(a)   &=  \gen{ \theta(a) \cap \cz{G}{b} }{ b \in B\nonid } \\
                    &=  \gen{ \cz{G}{a} \cap \theta(b) }{ b \in B\nonid } \\
                    &=  \gen{ \cz{G}{a} \cap \cz{K}{b} }{ b \in B\nonid }  = \cz{K}{a}.
    \end{align*}
    The conclusion follows.
\end{proof}

Henceforth we assume in addition to Hypothesis~\ref{er:1} that \[
    \mbox{\it $A$ is an elementary abelian $r$-group for some prime $r$.}
\]
Recall (see \cite{PFsft} for example)
that if $1 \not= B \leq A$ then $\theta(B)$ is defined by \[
    \theta(B) = \bigcap_{b \in B\nonid}\theta(b).
\]
Moreover if $H$ is a $\theta$-subgroup then $\cz{H}{B} = H \cap \theta(B)$.

\begin{Definition}\label{er:3}
    \[
    \norm{\theta} = \card{\theta(A)}\prod_{B \in \hyp{A}} \card{\theta(B):\theta(A)}.
    \]
\end{Definition}
\noindent Note that $\norm{\theta} < \infty$ since by the definition of signalizer functor,
the subgroups $\theta(a)$ are finite.
The definition is motivated by the following:

\begin{Theorem}[The Wielandt Order Formula]
    Suppose that $A$ acts coprimely on the group $H$.
    Then \[
        \card{H} = \card{\cz{H}{A}} \prod_{B \in \hyp{A}} \card{\cz{H}{B}:\cz{H}{A}}.
    \]
\end{Theorem}

\begin{Lemma}\label{er:5}
    \begin{enumerate}
        \item[(a)]  Let $\psi$ be a subfunctor of $\theta$.
                    Then $\norm{\psi} \leq \norm{\theta}$
                    with equality if and only if $\psi = \theta$.

        \item[(b)]  Let $H$ be a $\theta$-subgroup of $G$.
                    Then $\card{H} \leq \norm{\theta}$ with equality
                    if and only if $\theta(a) \leq H$ for all $a \in A\nonid$.

        \item[(c)]  If $\theta$ is complete then $\card{\theta(G)} = \norm{\theta}$.

        \item[(d)]  Suppose that $N$ is a normal $\theta$-subgroup of $G$.
                    Set $\br{G} = G/N$ and define $\br{\theta}$ by \[
                        \br{\theta}(a) = \br{\theta(a)}
                    \]
                    for all $a \in A\nonid$.
                    Then:
                    \begin{enumerate}
                        \item[(i)]  $\br{\theta}$ is an $A$-signalizer functor on $\br{G}$.

                        \item[(ii)] $\br{\theta}(B) = \br{\theta(B)}$ for all $1 \not= B \leq A$.

                        \item[(iii)]    $\theta$ is complete if and only if $\br{\theta}$ is complete.

                        \item[(iv)] $\norm{\br{\theta}} \leq \norm{\theta}$ with equality
                                    if and only if $N = 1$.
                    \end{enumerate}

        \item[(e)]  Let $a \in A\nonid$.
                    Then $\theta(a) = \gen{ \theta(B) }{ a \in B \in \hyp{A} }$.
    \end{enumerate}
\end{Lemma}
\begin{proof}
    This follows from Coprime Action and the Wielandt Order Formula.
\end{proof}

Finally we develop an idea of McBride that results in a fundamental dichotomy
in the proof of the Signalizer Functor Theorem.

\begin{Definition} \label{er:6} \mbox{}
    \begin{itemize}
        \item   $\theta$ is \emph{semisimple} if $1$ is the only $\theta(A)$-invariant
                solvable $\theta$-subgroup.

        \item   $\theta$ is \emph{nearsolvable} if $\theta(A)$ is solvable
                and every composition factor of every proper $\theta$-subgroup
                is isomorphic to \badfourlist.
    \end{itemize}
\end{Definition}
\noindent McBride's idea was to separate out the nonsolvable pieces of $\theta$
from the solvable pieces.
This is not possible -- but it nearly is.
The difficulty arises because the groups listed possess an automorphism
of order $r$ whose fixed point subgroup is solvable.
The following result is \cite[Theorem~6.6]{McB1},
a presentation of which may also be found in \cite[Theorem~8.8]{PFII}.

\begin{Theorem}[McBride's Dichotomy] \label{er:7}
    Suppose that $\theta$ is a minimal counterexample to the Signalizer Functor Theorem.
    Then $\theta$ is either semisimple or nearsolvable.
\end{Theorem}

\section{The minimal counterexample} \label{m}
Henceforth we assume the Signalizer Functor Theorem
to be false and let $(A,G,\theta)$ to be a counterexample.
By Lemma~\ref{er:2} we may suppose that $A$ is an elementary
abelian $r$-group with rank 3 for some prime $r$.
Then we may assume that $\norm{\theta}$ has been minimized.
Without loss
\begin{equation} \tag{$1$}
    G = \gen{ \theta(a) }{ a \in A\nonid }.
\end{equation}

In broad outline,
the proof proceeds as follows:
show that the family of subgroups $\set{ \theta(a) }{ a \in A\nonid }$
resembles the family of centralizers $\set{ \cz{G^{*}}{a} }{ a \in A\nonid }$
of some $A$-simple group $G^{*}$.
Then invoke a suitable characterization theorem,
namely Theorem~\ref{as:2}.

Most of the difficulty lies in establishing \[
    \layerr{\theta(a)} \not= 1
\]
for some $a \in A\nonid$ and then that \[
    \layerr{\layerr{\theta(a)} \cap \cz{G}{b}} \leq \layerr{\theta(b)}
\]
for all $a,b \in A\nonid$.

We define some notation:
\begin{itemize}
    \item   $\Theta$ is the set of proper $\theta$-subgroups of $G$.

    \item   $\fancyL$ is the set of $\theta(A)$-invariant members of $\Theta$.

    \item   $\Theta^{*}$ and $\fancyLstar$ denote the sets of maximal members
            of $\Theta$ and $\fancyL$ respectively.
\end{itemize}

\noindent Note that it could be the case that $G$ is itself a $\theta$-subgroup,
but in that case, $G$ is not a $K$-group.

If $H$ is an $A$-invariant subgroup of $G$ and of the $\theta$-subgroups of $G$ contained in $H$
there is a unique maximal one,
then we denote that $\theta$-subgroup by \[
    \theta(H)
\]
and say that \emph{$\theta(H)$ is defined.}
If $X \leq G$ then we abbreviate $\n{G}{X}$ and $\cz{G}{X}$ to $\nn{X}$ and $\cc{X}$ respectively.

\begin{Lemma}\label{m:1}
    \begin{itemize}
        \item[(a)]  The members of $\Theta$ are $K$-groups.

        \item[(b)]  Every member of $\Theta$, resp. $\fancyL$,
                    is contained in a member of $\Theta^{*}$, resp. $\fancyLstar$.

        \item[(c)]  If $H$ is a proper $A$-invariant subgroup of $G$ then
                    $\theta(H)$ is defined and $\theta(H) \in \Theta$.

        \item[(d)]  If $1 \not= H \in \Theta$ then $\nn{H} \not= G$.

        \item[(e)]  If $M \in \Theta^{*}$ then $M = \theta(\nn{X})$
                    for all $1 \not= X \characteristic M$.

        \item[(f)]  If $M,L \in \Theta^{*}$ then $\n{L}{X} \leq M$
                    for all $1 \not= X \characteristic M$.

        \item[(g)]  $\theta(A)$ is contained in every member of $\fancyLstar$
                    and $\fancyLstar \subseteq \Theta^{*}$.
    \end{itemize}
\end{Lemma}
\begin{proof}
    This is a consequence of Lemma~\ref{er:5}, $(1)$ and the minimality of $\norm{\theta}$.
\end{proof}

\begin{Corollary}\label{m:2}
    Let $M,N \in \Theta^{*}$.
    Suppose that $\ostar{M} \leq N$ and $\ostar{N} \leq M$.
    Then $M = N$.
\end{Corollary}
\begin{proof}
    We have $\ostar{M} \leq M \cap N \leq N$ so Lemma~\ref{p:2}(f)
    implies $\ostar{M} = \ostar{M \cap N}$.
    Similarly $\ostar{N} = \ostar{M \cap N}$.
    Then $\ostar{M} = \ostar{N}$.
    If $\ostar{M} \not= 1$ then the conclusion follows from Lemma~\ref{m:1}(e).
    If $\ostar{M} = 1$ then $M = N = 1$ and again the conclusion holds.
\end{proof}

If $M,N \in \fancyLstar$ then $M$ and $N$ are comaximal by Lemma~\ref{m:1}(f).
However, a little more can be said.

\begin{Lemma}\label{m:3}
    Let $M,N \in \fancyLstar, B \in \hyp{A}$ and $x \in \theta(B)$.
    Then $M$ and $N^{x}$ are comaximal.
\end{Lemma}
\begin{proof}
    Suppose $1 \not= T \characteristic M$ with $T \leq M \cap N^{x}$.
    Set $L = \n{N^{x}}{T}$.
    Let $b \in B\nonid$.
    Then \[
        \cz{L}{b} \leq \cz{M^{x}}{b} = \left( \cz{M}{b} \right)^{x} \leq \theta(b)^{x} = \theta(b)
    \]
    so $\cz{L}{b} \leq \theta(b) \cap \nn{T} \leq \theta(\nn{T}) = M$.
    Since $B$ is noncyclic,
    Coprime Action(d) implies $L \leq M$.
    Hence $M$ is maximal with respect to $N^{x}$.
    Similarly, $N^{x}$ is maximal with respect to $M$.
\end{proof}

Recall from \S\ref{psgps} that $\thetasol$ is defined by \[
    \thetasol(a) = \op{\mathrm{sol}}{\theta(a);A}
\]
for each $a \in A\nonid$,
where $\op{\mathrm{sol}}{\theta(a);A}$ is the largest
$A\cz{\theta(a)}{A}$-invariant  solvable subgroup of $\theta(a)$.
Note that $\cz{\theta(a)}{A} = \theta(A)$.
Theorem~\ref{psgps:5} implies that $\thetasol(a)$
is itself solvable and that $\thetasol$ is an $A$-signalizer functor on $G$.
The Solvable Signalizer Functor Theorem implies that $\thetasol$ is complete.
Let \[
    S = \thetasol(G).
\]

\begin{Lemma}\label{m:4}
    $S$ is the unique maximal $\theta(A)$-invariant solvable $\theta$-subgroup of $G$.
\end{Lemma}
\begin{proof}
    Apply Theorem~\ref{psgps:5}
\end{proof}

\noindent McBride's Dichotomy implies that if $S = 1$ then $\theta$ is semisimple
and if $S \not= 1$ then $\theta$ is nearsolvable.

This section concludes by eliminating a certain configuration.

\begin{Lemma}\label{m:6}
    The following is impossible: $e \in A\nonid$,
    $M \in \Theta$ and \[
        [\theta(a),e] \leq M
    \]
    for all $a \in A\nonid$.
\end{Lemma}
\begin{proof}
    Assume that it does hold.
    An argument of Bender, see \cite[Theorem~4.2]{PFsft},
    implies that $\theta(e)$ normalizes $[M,e]$ and that
    $\theta$ is complete with $\theta(G) = \theta(e)[M,e]$.
    Then $\theta(G)$ is a $K$-group because $\theta(e)$ and $[M,e]$ are $K$-groups.
    Lemma~\ref{psgps:6} implies that the composition factors of $\theta(G)$
    are to be found amongst the composition factors of the subgroups $\theta(a); a \in A\nonid$,
    contrary to $(A,G,\theta)$ being a counterexample to the Signalizer Functor Theorem.
\end{proof}
\begin{Corollary}\label{m:7}
    Suppose that $\psi$ is a subfunctor of $\theta$, $e \in A\nonid$
    and \[
        [\theta(a),e] \leq \psi(a)
    \]
    for all $a \in A\nonid$.
    Then $\psi = \theta$.
\end{Corollary}
\begin{proof}
    Suppose that $\psi \not= \theta$.
    Lemma~\ref{er:5}(a) implies that $\norm{\psi} < \norm{\theta}$
    so the minimality of $\norm{\theta}$ implies that $\psi$ is
    complete and that $\psi(G)$ is a $K$-group.
    Note that $\cz{\psi(G)}{a} = \psi(a) \leq \theta(a)$
    for all $a \in A\nonid$ so $\psi(G)$ is a $\theta$-subgroup.
    In particular,
    $\psi(G) \not= G$ as $(A,G,\theta)$ is a counterexample to the
    Signalizer Functor Theorem.
    Lemma~\ref{m:6}, with $\psi(G)$ in the role of $M$,
     supplies a contradiction.
\end{proof}

\section{Subfunctors}\label{sub}
Recall from \cite{PFsft} that if $p$ is a prime then
a $(p,\theta)$-subgroup is a $\theta$-subgroup that is also a $p$-group.
The collection of $(p,\theta)$-subgroups is partially ordered
by inclusion and its set of maximal elements is denoted by \[
    \syl{p}{G;\theta}.
\]
The Transitivity Theorem asserts that
$\theta(A)$ acts transitively on $\syl{p}{G;\theta}$.
In the proof of the Solvable Signalizer Functor Theorem it was
necessary to show that $\cz{A}{P} = 1$ whenever $1 \not= P \in \syl{p}{G;\theta}$.
This was accomplished using the subfunctor $\theta_{p'}$.
We shall extend those ideas to obtain information in the case
$\cz{A}{P} \not= 1$ and $\theta$ is nearsolvable.
First, a simple criterion for $\cz{A}{P}$ to be nontrivial.

\begin{Lemma}\label{sub:1}
    Let $e \in A\nonid$ and suppose $\theta(e) \leq M \in \fancyLstar$.
    Assume $p \in \primes{\ff{M}}$ and $[M,e]$ is a $p'$-group.
    Then $e$ centralizes every $(p,\theta)$-subgroup of $G$.
\end{Lemma}
\begin{proof}
    Choose $P \in \syl{p}{M;A}$.
    Then $[P,e] = 1$.
    Also, $1 \not= \op{p}{M} \leq P$ so as $M \in \fancyLstar$
    we have $\theta(\cc{P}) \leq \theta(\nn{\op{p}{M}}) = M$.
    Choose $Q$ with $P \leq Q \in \syl{p}{G;\theta}$.
    By Coprime Action(e),
    $\n{Q}{P} = [\n{Q}{P},e](\n{Q}{P} \cap \cc{e})$.
    Now $[P,e] = 1$ so $[\n{Q}{P},e] \leq \cz{Q}{P} \leq \theta(\cc{P}) \leq M$.
    Also $\n{Q}{P} \cap \cc{e} \leq \theta(e) \leq M$ whence $\n{Q}{P} \leq M$.
    Since $P \in \syl{p}{M}$ this forces $\n{Q}{P} = P$
    and then $P = Q \in \syl{p}{G;\theta}$.
    As $[P,e] = 1$ and $\theta(A) \leq \cc{e}$,
    the Transitivity Theorem implies that $e$ centralizes
    every member of $\syl{p}{G;\theta}$.
    The conclusion follows.
\end{proof}

Recall that if $p$ is a prime and $X$ is a group then $\op{p-\mathrm{sol}}{X}$
is the largest normal $p$-solvable subgroup of $X$.
Theorem~\ref{psgps:5} asserts that the map $\thetapsol$ defined by
\begin{align*}
    \thetapsol(a)   &= \op{p-\mathrm{sol}}{\theta(a);A} \\
                    &= \mbox{the unique maximal $A\theta(A)$-invariant $p$-solvable subgroup of $\theta(a)$}
\end{align*}
is a subfunctor of $\theta$.
Similarly so is the map $\theta_{p'}$ defined by
\begin{align*}
    \theta_{p'}(a) &= \op{p'}{\theta(a);A}.
\end{align*}

\noindent We state the main result of this section.

\begin{Theorem}\label{sub:2}
    Assume the following:
    \begin{itemize}
        \item   $p \in \primes{\theta}$.

        \item   $e \in A\nonid$ and $e$ centralizes every $(p,\theta)$-subgroup of $G$.

        \item   $\theta$ is nearsolvable.
    \end{itemize}
    Then the following hold:
    \begin{enumerate}
        \item[(a)]  $G$ possesses a unique maximal $\theta(A)$-invariant
                    $p$-solvable $\theta$-subgroup.

        \item[(b)] For all $a \in A\nonid$, \[
                        \theta(a) = [\theta(a),e]\thetapsol(a).
                    \]
                    In particular,
                    $\theta(e)$ is $p$-solvable.
    \end{enumerate}
\end{Theorem}

\begin{proof}[Proof of Theorem~\ref{sub:2}(a)]
    Assume that $\thetapsol = \theta$.
    Let $a \in A\nonid$.
    By hypothesis, $e$ centralizes every $A$-invariant Sylow $p$-subgroup
    of $\theta(a)$ so using Coprime Action(h) we have \[
        [\theta(a),e] \leq \op{p'}{\theta(a)} \leq \theta_{p'}(a).
    \]
    Corollary~\ref{m:7} implies that $\theta = \theta_{p'}$.
    But then $\theta(a)$ is a $p'$-group for all $a \in A\nonid$,
    contrary to $p \in \primes{\theta}$.
    We deduce that $\thetapsol \not= \theta$.
    The minimality of $\norm{\theta}$ implies that $\thetapsol$
    is complete and that $\thetapsol(G)$ is $p$-solvable.
    Theorem~\ref{psgps:5} implies that $\thetapsol(G)$ is the
    unique maximal $\theta(A)$-invariant $p$-solvable $\theta$-subgroup.
\end{proof}

\begin{Lemma}\label{sub:3}
    Assume the hypotheses of Theorem~\ref{sub:2}.
    Let $X \in \fancyL$ and suppose that $X = [X,e]$.
    Set $\br{X} = X/\op{p-\mathrm{sol}}{X}$.
    \begin{enumerate}
        \item[(a)]  $\gfitt{\br{X}} = \layerr{\br{X}}, \zz{\layerr{\br{X}}} = 1$
                    and each component of $\br{X}$ is normalized but not
                    centralized by $e$.

        \item[(b)]  $\cz{X}{e}$ is $p$-solvable.
    \end{enumerate}
\end{Lemma}
\begin{proof}
    Since $\op{p-\mathrm{sol}}{\br{X}} = 1$ we have
    $\zz{\layerr{\br{X}}} = \ff{\br{X}} = \op{p'}{\br{X}} = 1$.
    In particular,
    each component of $\br{X}$ has order divisible by $p$.
    By hypothesis,
    $e$ centralizes a Sylow $p$-subgroup of $\layerr{\br{X}}$.
    Then $e$ acts trivially on $\compp{\br{X}}$.
    Since $\br{X} = [\br{X},e]$ it follows that each component
    of $\br{X}$ is normal in $\br{X}$.
    If $\br{K} \in \compp{\br{X}}$ and $[\br{K},e] = 1$
    then $\br{X} = [\br{X},e]$ centralizes $\br{K}$,
    a contradiction.
    Thus (a) holds.

    Let $\br{K} \in \compp{\br{X}}$.
    Since $\theta$ is nearsolvable, $\br{X}$ is nearsolvable and
    $\br{K} \isomorphic $ \badfourlist.
    Now $e$ induces a nontrivial automorphism of order $r$
    on $\br{K}$ and $\br{K}$ is an $r'$-group.
    It follows that $\cz{\br{K}}{e}$ is solvable.
    Then $\cz{\layerr{\br{X}}}{e}$ is solvable.
    By (a) and the Schreier Property,
    $\br{X}/\layerr{\br{X}}$ is solvable.
    Now $\br{X} = X/\op{p-\mathrm{sol}}{X}$, whence $\cz{X}{e}$
    is $p$-solvable and (b) holds.
\end{proof}

\begin{proof}[Proof of Theorem~\ref{sub:2}(b)]
    For each $a \in A\nonid$ define $\psi(a)$ by \[
        \psi(a) = [\theta(a),e]\thetapsol(a).
    \]
    We claim that $\psi$ is an $A$-signalizer functor.
    Indeed, let $a,b \in A\nonid$.
    Coprime Action(e) implies \[
        \psi(a) \cap \cc{b} = ([\theta(a),e] \cap \cc{b})(\thetapsol(a) \cap \cc{b}).
    \]
    Let $Y = [\theta(a),e] \cap \cc{b}$.
    Then $Y \leq \theta(a) \cap \cc{b} \leq \theta(b)$.
    Now $Y = [Y,e]\cz{Y}{e}$ by Coprime Action(c) and $[Y,e] \leq \psi(b)$.
    Lemma~\ref{sub:3}(b) implies that $[\theta(a),e] \cap \cc{e}$ is $p$-solvable.
    Then $\cz{Y}{e}$ is an $A\theta(A)$-invariant $p$-solvable subgroup of $\theta(b)$,
    so $\cz{Y}{e} \leq \thetapsol(b) \leq \psi(b)$.
    Thus $Y \leq \psi(b)$.
    As $\thetapsol$ is an $A$-signalizer functor we have
    $\thetapsol(a) \cap \cc{b} \leq \thetapsol(b) \leq \psi(b)$.
    Hence $\psi(a) \cap \cc{b} \leq \psi(b)$ and the claim is established.
    Corollary~\ref{m:7} implies $\psi = \theta$.
    Also,
    $\theta(e) = \psi(e) = \thetapsol(e)$ so $\theta(e)$ is $p$-solvable.
\end{proof}

\section{The First Uniqueness Theorem} \label{fut}
The aim of this section is to prove a result that deals with the
characteristic $p$ case arising in conclusion (c) of Bender's Maximal Subgroup Theorem.
First we recall the following: let $p$ be a prime.
\begin{itemize}
    \item   A group $M$
            has \emph{characteristic $p$} if $\gfitt{M} = \op{p}{M}$.

    \item   If $A$ acts coprimely in the group $M$ then $\op{p}{M;A}$
            is the intersection of all $A$-invariant Sylow $p$-subgroups of $M$.
            It is the unique maximal $A\cz{M}{A}$-invariant $p$-subgroup of $M$.

    \item   $\op{p}{G;\theta}$ is the intersection of the members of $\syl{p}{G;\theta}$.
            It is the unique maximal $\theta(A)$-invariant $(p,\theta)$-subgroup.
\end{itemize}
The uniqueness assertions follow from Coprime Action(a) and the Transitivity Theorem.

\begin{Theorem}[The First Uniqueness Theorem] \label{fut:1}
    Let $p$ be a prime and suppose $M \in \fancyLstar$ has characteristic $p$.
    \begin{enumerate}
        \item[(a)]  $\op{p}{G;\theta} \leq M$.

        \item[(b)]  $M$ is the only member of $\fancyLstar$ with characteristic $p$.
    \end{enumerate}
\end{Theorem}

\begin{proof}[Proof of Theorem~\ref{fut:1}]
    Let $N = \theta(\nn{\op{p}{M;A}})$.
    Note that $\op{p}{M} \leq \op{p}{M;A}$ so $\op{p}{M;A} \not= 1$
    since $M$ has characteristic $p$.
    Let $B \in \hyp{A}$ and $x \in \cz{N}{B}$.
    Then \[
        \op{p}{M;A} = \op{p}{M;A}^{x} = \op{p}{M^{x};A^{x}}.
    \]
    Using Lemma~\ref{m:3} we see that the hypotheses of Theorem~\ref{ben:3}
    are satisfied with $M_{1} = M, A_{1} = A, M_{2} = M^{x}$ and $A_{2} = A^{x}$.
    Consequently $M = M^{x}$.
    Then $x \in \theta(B) \cap \nn{M} \leq \theta(\nn{M}) = M$.
    We deduce that $\cz{N}{B} \leq M$ for all $B \in \hyp{A}$.
    Coprime Action(d) implies $N \leq M$.

    Let $P = \op{p}{G;\theta}$.
    Then $P$ contains every $\theta(A)$-invariant $(p,\theta)$-subgroup.
    Also $M \in \fancyLstar$ so $\theta(A) \leq M$,
    in fact $\theta(A) = \cz{M}{A}$.
    It follows that \[
        P \cap M = \op{p}{M;A}.
    \]
    Then $\n{P}{P \cap M} \leq \theta(\nn{\op{p}{M;A}}) = N \leq M$
    so $\n{P}{P \cap M} \leq P \cap M$.
    This forces $P = P \cap M \leq M$ so $P = \op{p}{M;A}$ and (a) holds.

    To prove (b),
    suppose $N \in \fancyLstar$ also has characteristic $p$.
    Then $\op{p}{M;A} = P = \op{p}{N;A}$.
    Another application of Theorem~\ref{ben:3} forces $M = N$.
\end{proof}

\section{The subgroups $M_{a}$} \label{ma}
The main result of this section is the following:

\begin{Theorem} \label{sm:1}
    Let $a \in A\nonid$.
    There exists $M_{a}$ such that the following hold:
    \begin{enumerate}
        \item[(a)]  $\theta(a) \leq M_{a} \in \fancyLstar$.

        \item[(b)]  If $N \in \fancyLstar$ satisfies \[
                        M_{a} \linksto N \quad\text{and}\quad \theta(a) \leq N
                    \]
                    then $M_{a} = N$.

        \item[(c)]  If $X \not= 1$ is an $A\theta(a)$-invariant subnormal
                    subgroup of $M_{a}$ then $\theta(\nn{X}) \leq M_{a}$.
    \end{enumerate}
\end{Theorem}

\noindent Throughout the remainder of this paper,
we let $\set{M_{a}}{a \in A\nonid}$ be the family of subgroups constructed
in Theorem~\ref{sm:1}.

It is a trivial consequence of Coprime Action and the fact that
$G = \gen{\theta(a)}{a \in A\nonid}$ that if $B \in \hyp{A}$
then there exists $b,b' \in B\nonid$ with $M_{b} \not= M_{b'}$.
In fact, we can go a little further.

\begin{Lemma}\label{sm:2}
    Let $B \in \hyp{A}$.
    \begin{enumerate}
        \item[(a)]  Let $B_{0} \in \hyp{B}$.
                    Then $M_{b}$ takes at least two values as $b$ ranges over $B \setminus B_{0}$.

        \item[(b)]  $M_{a}$ takes at least three values as $a$ ranges over $A \setminus B$.
    \end{enumerate}
\end{Lemma}

\begin{proof}[Proof of Theorem~\ref{sm:1}]
    Choose $W$ maximal subject to \[
        \text{$W \in \fancyL$, $W$ is $\theta(a)$-invariant and $W = \layerr{W} = [W,a]$.}
    \]
    If $W \not= 1$ choose $M$ with $\theta(\nn{W}) \leq M \in \fancyLstar$.
    If $W = 1$ choose $M$ with $\theta(a) \leq M \in \fancyLstar$ and
    if possible with $\cz{\op{p}{M}}{a} = 1$ for some $p \in \primes{\ff{M}}$.
    In both cases,
    $\theta(a) \leq M$ and $\theta(a) = \cz{M}{a}$.
    Moreover, $W$ is $A\cz{M}{a}$-invariant so as $W = \layerr{W} = [W,a]$,
    Theorem~\ref{aut:3}(a) implies $W \leq \layerr{M}$.

    Suppose $N$ satisfies \[
        \theta(a) \leq N \in \fancyLstar \quad\text{and}\quad M \linksto N.
    \]
    We will prove that
    \begin{equation} \tag{$*$}
        M = N.
    \end{equation}
    Since $M \linksto N$ we have $W \leq \layerr{M} \leq N$ by Theorem~\ref{ben:2}(a)
    so another application of Theorem~\ref{aut:3}(a) implies $W \leq \layerr{N}$.
    Then $W = [W,a] \leq [\layerr{N},a]$.
    Now $\theta(a) \leq N$ so $[\layerr{N},a]$ is $\theta(a)$-invariant.
    It is also normal in $\layerr{N}$ so it is the central product
    of its components.
    The maximal choice of $W$ forces \[
        W = [\layerr{N},a] \normal \gfitt{N}.
    \]
    Suppose $W \not= 1$.
    Then $\gfitt{N} \leq \theta(\nn{W}) \leq M$ so $N \linksto M$.
    Since $\layerr{N} \not= 1$,
    Theorem~\ref{ben:2}(c) forces $M = N$.
    Hence we may assume that $W = 1$.
    In particular \[
        \layerr{N} \leq \theta(a) \leq M.
    \]

    We claim that \[
        \primes{\ff{N}} \subseteq \primes{\ff{M}}.
    \]
    Assume false and choose $q \in \primes{\ff{N}} \setminus \primes{\ff{M}}$.
    Theorem~\ref{ben:2}(a) implies $M \cap \op{q}{N} = 1$.
    Now $\cz{\op{q}{N}}{a} \leq \theta(a) \cap \op{q}{N} \leq M \cap \op{q}{N}$
    so $\cz{\op{q}{N}}{a} = 1$.
    Recall that $\theta(a) \leq N$.
    The choice of $M$ implies that there exists $p \in \primes{\ff{M}}$
    with $\cz{\op{p}{M}}{a} = 1$.
    As $M \linksto N$ we have $\zz{\op{p}{M}} \leq N$.
    Set $X = \zz{\op{p}{M}}\op{q}{N}$.
    Then $\zz{\op{p}{M}}$ is an $\listgen{a}\cz{X}{a}$-invariant subgroup of $X$
    and Coprime Action(c) implies $\zz{\op{p}{M}} = [\zz{\op{p}{M}},a]$.
    Theorem~\ref{aut:2}(a) implies $\zz{\op{p}{M}} \leq \op{p}{X}$
    whence $\op{q}{N} \leq \theta(\nn{\zz{\op{p}{M}}}) = M$,
    contrary to $M \cap \op{q}{N} = 1$.
    The claim is established.

    Theorem~\ref{ben:2}(c) and the First Uniqueness Theorem
    imply $M = N$,
    which proves $(*)$.

    Suppose $1 \not= X \subnormal M$ is $A\theta(a)$-invariant.
    Choose $N$ with $\theta(\nn{\gfitt{X}}) \leq N \in \fancyLstar$.
    Now $\gfitt{X} \subnormal \gfitt{M}$ so $M \linksto N$.
    Then $M = N$ by $(*)$ and
    so $\theta(\nn{X}) \leq \theta(\nn{\gfitt{X}}) \leq M$.

    Set $M_{a} = M$ to complete the proof.
\end{proof}

\begin{proof}[Proof of Lemma~\ref{sm:2}]
    (a). Recall that $\rank{A}  = 3$ so $B_{0}$ is cyclic.
    Let $e$ be a generator for $B_{0}$.
    Assume the result is false and let $M$ denote the common value of $M_{b}$
    as $b$ ranges over $B \setminus B_{0}$.
    By Coprime Action(d),
    for each $a \in A\nonid$,
    \begin{align*}
        [\theta(a),e] &= \gen{[\theta(a) \cap \cc{b},e]}{ b \in B\setminus B_{0} } \\
                      &\leq \gen{ \theta(b) }{b \in B\setminus B_{0} } \leq M.
    \end{align*}
    Lemma~\ref{m:6} supplies a contradiction.

    (b). Assume the result is false.
    Then there exist $M,L$ with $M_{a} \in \listset{M,L}$ for all $a \in A \setminus B$.
    Let $a \in A \setminus B$.
    Then \[
        \theta(a) = \gen{ \theta(D) }{ a \in D \in \hyp{A} }.
    \]
    If $a \in D \in \hyp{A}$ then $D \not= B$ so $D \cap B \in \hyp{D}$.
    By (a),
    with $D$ in the role of $B$,
    $M_{d}$ takes at least two values as $d$ ranges over $D \setminus D \cap B$.
    Hence $M_{d} = M$ for some $d \in D \setminus D \cap B$.
    Consequently $\theta(D) \leq \theta(d) \leq M$ and we deduce that \[
            \theta(a) \leq M
    \]
    for all $a \in A \setminus B$.

    Choose $D \in \hyp{A}$ with $D \not= B$.
    Set $D_{0} = D \cap B \in \hyp{D}$ and let $e$ be a generator for $D_{0}$.
    Let $T$ be any $\theta$-subgroup.
    By Coprime Action(d), \[
        [T,e] = \gen{ [\cz{T}{d},e] }{ d \in D \setminus D_{0} } \leq M.
    \]
    In particular, $[\theta(a),e] \leq M$ for all $a \in A\nonid$.
    Again, Lemma~\ref{m:6} supplies a contradiction.
\end{proof}

\section{The Fermat case}
Since $S$ is the unique maximal $\theta(A)$-invariant solvable $\theta$-subgroup
it follows that $\ff{M_{a}} \leq S$ for all $a \in A\nonid$.
The goal of this section is to prove:

\begin{Theorem}\label{f:1}
    Let $a \in A\nonid$ and suppose $\layerr{M_{a}} = 1$.
    Then $[\ff{M_{a}},a]\ff{S}$ is nilpotent.
\end{Theorem}

\noindent In the case that $r$ is not a Fermat prime,
this follows readily from Theorem~\ref{aut:2}(a),
with $[\ff{M_{a}},a]$ in the role of $H$.
Just as in the author's proof of the Solvable Signalizer Functor Theorem,
the Fermat case requires special treatment.

Throughout the remainder of this section we assume Theorem~\ref{f:1}
to be false.
Theorem~\ref{aut:2}(a) implies $r$ is Fermat and that
$[\op{2}{M_{a}},a]\ff{S}$ is not nilpotent.
Set \[
    Q = [\op{2}{M_{a}},a]
\]
and choose an odd prime $p$ such that \[
    [\op{p}{S},Q] \not= 1.
\]

\begin{Lemma}\label{f:2}
    \begin{enumerate}
        \item[(a)]  $\op{p}{M_{a}} = 1, M_{a} \cap \op{p}{S} = 1$
                    and $\cz{\op{p}{S}}{a} = 1$.

        \item[(b)]  $Q = [Q,a]$ and $Q' = \frat{Q} = \zz{Q} = \cz{Q}{a}$.
    \end{enumerate}
\end{Lemma}
\begin{proof}
    (a). Since $\ff{M_{a}} \leq S$ and $\layerr{M_{a}} = 1$ we have $M_{a} \linksto S$.
    Suppose $p \in \primes{\ff{M_{a}}}$.
    Then $\listset{2,p} \subseteq \primes{\ff{M_{a}}}$ so Theorem~\ref{ben:2}(b)
    implies $\op{p}{S} \leq M_{a}$.
    Then $[\op{p}{S},Q] \leq \op{p}{S} \cap \op{2}{M_{a}} = 1$,
    a contradiction.
    Thus $p \not\in \primes{\ff{M_{a}}}$.
    Theorem~\ref{ben:2}(a) implies $M_{a} \cap \op{p}{S} = 1$.
    Finally $\cz{\op{p}{S}}{a} \leq \theta(a) \cap \op{p}{S} \leq M_{a} \cap \op{p}{S} = 1$.

    (b). The first assertion is Coprime Action(c) and the second is Coprime Action(i)
    provided we can show $[U,a] = 1$ whenever $U$ is a characteristic abelian subgroup of $Q$.
    Assume this to be false.
    Now $[U,a] \leq \op{2}{M_{a}}$ so $U$ is an $A\theta(a)$-invariant subnormal subgroup
    of $M_{a}$ and Theorem~\ref{sm:1} implies $\theta(\nn{[U,a]}) \leq M_{a}$.
    On the other hand,
    Lemma~\ref{p:6}(c) implies $[U,a]$ acts trivially on $\op{p}{S}$.
    Hence $\op{p}{S} \leq \theta(\nn{[U,a]})$ contrary to
    $M_{a} \cap \op{p}{S} = 1$ which completes the proof.
\end{proof}

Set \[
    W = \gen{ [\cz{Q}{B},a]' }{ B \in \hyp{A} }.
\]

\begin{Lemma}\label{f:3}
    \begin{enumerate}
        \item[(a)]  $W \leq \zz{Q}$ and $1 \not= W \normal \theta(A)$.

        \item[(b)]  If $H \in \fancyLstar$ then $W \leq \sol{H}$.

        \item[(c)]  $\cz{\op{p}{S}}{W} = 1$.
    \end{enumerate}
\end{Lemma}
\begin{proof}
    Lemma~\ref{f:2}(b) implies $W \leq Q' = \zz{Q} = \cz{Q}{a}$.
    Let $B \in \hyp{A}$,
    set $Q_{0} = [\cz{Q}{B},a]$ and suppose $Q_{0} \not= 1$.
    Then $a \not\in B$ and $A = \listgen{B,a}$.
    As $Q_{0}' \leq W \leq \cz{Q}{a}$ we obtain $[Q_{0}',A] = 1$.
    Moreover $\theta(A)$ normalizes $Q,B$ and $a$ so $Q_{0}$
    is $\theta(A)$-invariant and $Q_{0}' \normal \theta(A)$.
    Consequently $W \normal \theta(A)$.

    For each $b \in B\nonid$ we have $Q_{0} \leq \op{2}{\theta(b);A}$
    and then Lemma~\ref{p:6}(a) implies $Q_{0}' \leq \sol{\theta(b)}$.
    Let $H \in \fancyLstar$.
    Then $Q_{0}' \leq \theta(A) \leq H$.
    Using Lemma~\ref{p:6}(b) for the last containment,
    we have \[
        Q_{0}' \leq \bigcap_{b \in B\nonid} \sol{\theta(b)} \cap H %
        \leq \bigcap_{b \in B\nonid} \sol{\cz{H}{b}} \leq \sol{H}.
    \]
    This proves (b).

    By Coprime Action(d),
    \begin{equation}\tag{$*$}
        Q = \gen{ [\cz{Q}{B},a] }{ B \in \hyp{A} }.
    \end{equation}
    As $[\op{p}{S},Q] \not= 1$ we may choose $B$ with $[\op{p}{S},Q_{0}] \not= 1$.
    Lemma~\ref{p:6}(c),
    with $Q_{0}$ in the role of $X$,
    implies that $Q_{0}$ is nonabelian.
    Then $1 \not= Q_{0}' \leq W$,
    which completes the proof of (a).

    To prove (c) consider the action of $Q_{0}$ on $\cz{\op{p}{S}}{W}$.
    Note that $\cz{\op{p}{S}}{W}$ is $Q$-invariant because $W \leq \zz{Q}$.
    Now $Q_{0}' \leq W$ so $Q_{0}$ induces an abelian group on $\cz{\op{p}{S}}{W}$.
    Lemma~\ref{p:6}(c) implies $[\cz{\op{p}{S}}{W},Q_{0}] = 1$.
    Then $(*)$ implies $[\cz{\op{p}{S}}{W},Q] = 1$.
    Now $Q$ is an $A\theta(a)$-invariant subnormal subgroup of $M_{a}$
    so Theorem~\ref{sm:1} implies $\theta(\nn{Q}) \leq M_{a}$.
    Then $\cz{\op{p}{S}}{W} \leq M_{a} \cap \op{p}{S} = 1$.
\end{proof}

\begin{Lemma}\label{f:4}
    Let $H \in \fancyLstar$ and suppose $\op{p}{S} \cap H \not= 1$.
    Then $\op{p}{S} \normal \op{p}{H}$.
\end{Lemma}
\begin{proof}
    Set $P = \op{p}{S} \cap H$.
    Lemma~\ref{f:3}(b),(c) and Coprime Action(c) imply \[
        P = [P,W] \leq \sol{H}.
    \]
    Now $H \in \fancyLstar$ so $\sol{H} \leq S$ and then
        $1 \not= P \leq \op{p}{\sol{H}} \leq \op{p}{H}.$
    Also $\op{p}{H} \leq S$ so $\op{p}{H}\op{p}{S}$ is a $p$-group.
    Since \[
        \n{\op{p}{S}}{\op{p}{H}} \leq \op{p}{S} \cap H = P \leq \op{p}{H}
    \]
    it follows that $\op{p}{H}\op{p}{S} = \op{p}{H}$.
    As $\op{p}{H} \leq S$, the conclusion follows.
\end{proof}

Choose $N$ with \[
    \theta(\nn{\op{p}{S}}) \leq N \in \fancyLstar.
\]

\begin{Lemma}\label{f:5}
    $\layerr{N} = 1$.
\end{Lemma}
\begin{proof}
    We have $Q \leq N$ whence $Q \leq \op{2}{N;A}$.
    Lemma~\ref{p:6}(a) implies $Q' \leq \sol{N}$,
    so $[Q', \layerr{N}] = 1$.
    Now $Q' \not= 1$ so Theorem~\ref{sm:1} implies $\theta(\nn{Q'}) \leq M_{a}$,
    whence $\layerr{N} \leq M_{a}$.
    Now $\ff{M_{a}} \leq S \leq N$ and $\layerr{N}$ and $\ff{M_{a}}$
    normalize each other.
    Then $[\layerr{N},\ff{M_{a}}] = 1$.
    By hypothesis, $\layerr{M_{a}} = 1$ so it follows that $\layerr{N} = 1$.
\end{proof}

We can now complete the proof of Theorem~\ref{f:1}.
By Coprime Action(d) there exists $B \in \hyp{A}$ with
$\cz{\op{p}{S}}{B} \not= 1$.
By Lemma~\ref{sm:2} there exists $b \in B\nonid$ with $M_{b} \not= N$.
Now $1 \not= \cz{\op{p}{S}}{B} \leq \theta(b) \cap \op{p}{S} \leq M_{b} \cap \op{p}{S}$
so Lemma~\ref{f:4} implies $\op{p}{S} \normal \op{p}{M_{b}}$.
Then $\gfitt{M_{b}} \leq N$ and $M_{b} \linksto N$.
Since $\layerr{N} = 1$,
Theorem~\ref{ben:2} and the First Uniqueness Theorem
imply there exists a prime $t$ with \[
    \text{$\op{t}{N} \not= 1, \op{t}{M_{b}} = 1$ and $M_{b} \cap \op{t}{N} = 1$.}
\]
As $\theta(\cz{\op{t}{N}}{b}) \leq \theta(b) \leq M_{b}$ we also have \[
    \cz{\op{t}{N}}{b} = 1.
\]

\noindent Since $p \not= 2$,
Lemma~\ref{p:6}(c) implies $[\op{p}{M_{b}},b]$ centralizes $\op{t}{N}$.
If $[\op{p}{M_{b}},b] \not= 1$ then
$\op{t}{N} \leq \theta(\nn{[\op{p}{M_{b}},b]}) \leq M_{b}$ by Theorem~\ref{sm:1},
a contradiction.
Thus $[\op{p}{M_{b}},b] = 1$.
Then as $\op{p}{S} \leq \op{p}{M_{b}}$,
we have \[
    [M_{b},b] \leq \cz{M_{b}}{\op{p}{M_{b}}} \leq \theta(\nn{\op{p}{S}}) \leq N.
\]

Set $U = [M_{b},b]$.
Lemma~\ref{p:6}(c) implies that $U/\cz{U}{\op{t}{N}}$ is a solvable $\listset{2,t}$-group.
Let $V$ be the subgroup of $U$ generated by
$U^{(\infty)}$ and the $\listset{2,t}'$-elements of $U$.
Then $[\op{t}{N},V] = 1$.
Also $V \characteristic U \normal M_{b}$ so $V \normal M_{b}$.
If $V \not= 1$ then $\op{t}{N} \leq M_{b}$,
a contradiction.
Thus $V = 1$ and $U$ is a solvable $\listset{2,t}$-group.

Note that $p \in \primes{\ff{M_{b}}}$ since $\op{p}{S} \normal \op{p}{M_{b}}$
and so $p \not= t$ as $\op{t}{M_{b}} = 1$.
Also, $p \not= 2$.
Thus $[M_{b},b]$ is a $p'$-group.
McBride's Dichotomy, Lemma~\ref{sub:1} and Theorem~\ref{sub:2} imply
that there exists a unique maximal $\theta(A)$-invariant
$p$-solvable $\theta$-subgroup $K$
and that $\theta(b)$ is $p$-solvable.
Now $M_{b} = \cz{M_{b}}{b}U = \theta(b)U$.
Since $U$ is a normal solvable subgroup of $M_{b}$ we deduce that
$M_{b}$ is $p$-solvable and then that $M_{b} = K$.
But $\op{t}{N}$ is $p$-solvable and $\theta(A)$-invariant,
whence $\op{t}{N} \leq M_{b}$.
This contradiction completes the proof of Theorem~\ref{f:1}.

\section{The Second Uniqueness Theorem} \label{sut}
The goal of this section is the prove the following:
\begin{Theorem}[The Second Uniqueness Theorem] \label{sut:1}
    Let $a \in A\nonid$ and suppose that $\layerr{M_{a}} = 1$.
    Then:
    \begin{enumerate}
        \item[(a)]  $S \leq M_{a}$.

        \item[(b)]  If $b \in A\nonid$ and $\layerr{M_{b}} = 1$ then $M_{b} = M_{a}$.
    \end{enumerate}
\end{Theorem}

\begin{Lemma}\label{sut:0}
    Suppose $I$ and $J$ are subgroups of the group $X$.
    Suppose also that $\layerr{X} = 1$ and that $I\ff{X}$ and $J\ff{X}$ are nilpotent.
    Let $p$ and $q$ be distinct primes.
    Then $[\op{p}{I},\op{q}{J}] = 1$.
\end{Lemma}
\begin{proof}
    Set $Z = \zz{\ff{X}}$.
    Since $\layerr{X} = 1$ we have $\cz{X}{\ff{X}} = Z$.
    Now \[
        [\op{p}{I},\op{q}{J}] \leq [\cz{X}{\op{p'}{\ff{X}}}, \cz{X}{\op{q'}{\ff{X}}}] %
                              \leq \cz{\ff{X}} = Z.
    \]
    Hence $\op{p}{I}$ normalizes the nilpotent group $\op{q}{X}Z$.
    Then $[\op{p}{I},\op{q}{J}] \leq \op{q}{\op{q}{J}Z}$
    and the commutator is a $q$-group.
    Similarly, it is a $p$-group and hence is trivial.
\end{proof}

\begin{proof}[Proof of Theorem~\ref{sut:1}]
    Set $M = M_{a}$.
    Since $\layerr{M} = 1$ we have $\ff{M} \not= 1$ so
    McBride's Dichotomy implies that $\theta$ is nearsolvable.

    (a). Suppose first that $[\ff{M},a] = 1$.
    Coprime Action(g) implies $[M,a] = 1$.
    Thus $M = \theta(a)$.
    Choose $p \in \primes{\ff{M}}$.
    Theorem~\ref{sub:1} implies that there exists a unique maximal $\theta(A)$-invariant
    $p$-solvable $\theta$-subgroup and that $\theta(a)$,
    and hence $M$ is $p$-solvable.
    Since $M \in \fancyLstar$ it follows that $M$ is the said subgroup.
    Now $S$ is $\theta(A)$-invariant and solvable so $S \leq M$.
    Hence we may assume that $[\ff{M},a] \not= 1$.

    Note that \[
        \ff{M} \leq S
    \]
    because $\ff{M}$ is $\theta(A)$-invariant and solvable.
    In particular, $M \linksto S$.

    We claim that $\primes{\ff{S}} \subseteq \primes{\ff{M}}$.
    Indeed, suppose $q$ is a prime with $q \not\in \primes{\ff{M}}$.
    Using Theorems~\ref{f:1} and \ref{sm:1} we have
    $\op{q}{S} \leq \theta(\nn{[\ff{M},a]}) \leq M$.
    On the other hand,
    Theorem~\ref{ben:2}(a) implies $M \cap \op{q}{S} = 1$.
    Hence $\op{q}{S} = 1$ and the claim is established.

    Consider the case that $\card{\primes{\ff{M}}} \geq 2$.
    Theorem~\ref{ben:2}(b) implies \[
        \ff{S} \leq M.
    \]
    Since $\ff{M} \leq S$ we deduce that $\ff{M}\ff{S}$ is nilpotent.
    Let $x \in S$ and let $p,q \in \primes{\ff{M}}$ be distinct.
    Now $\ff{M}^{x}\ff{S}$ is nilpotent so Lemma~\ref{sut:0} implies $[\op{p}{M},\op{q}{M}^{x}] = 1$.
    Then $\op{q}{M}^{x} \leq \n{S}{\op{p}{M}} \leq M$.
    We deduce that $\ff{M}^{x} \leq M$ and so $M^{x} \linksto M$.
    If $x \in \cz{S}{B}$ for some $B \in \hyp{A}$ then Lemma~\ref{ben:3}
    implies that $M$ and $M^{x}$ are comaximal and so $M = M^{x}$ by Theorem~\ref{ben:2}(c)(ii).
    Then $x \in \n{S}{M} \leq M$ and we deduce that $\cz{S}{B} \leq M$
    for all $B \in \hyp{A}$.
    Coprime Action(d) forces $S \leq M$.
    Hence we may assume that \[
        \mbox{$\ff{M}$ is a $p$-group}
    \]
    for some prime $p$.

    The First Uniqueness Theorem implies $\op{p}{G;\theta} \leq M$
    so $\op{p}{G;\theta} = \op{p}{M;A}$.
    As $\op{p}{G;\theta}$ is $\theta(A)$-invariant and solvable
    we have $\op{p}{G;\theta} \leq S$
    and so $\op{p}{G;\theta} = \op{p}{S;A}$.
    Consequently $\op{p}{M;A} = \op{p}{S;A}$.
    Note that $\op{p}{M} \not= 1$ since $\ff{M}$ is a $p$-group.
    If $\op{p}{M}$ is abelian then Theorem~\ref{ben:4} implies
    $J(\op{p}{M}) = J(\op{p}{S})$.
    Then $S \leq \theta(\nn{J(\op{p}{M})}) = M$.
    Hence we may assume that $\op{p}{M}$ is nonabelian.

    Choose $N$ with $S \leq N \in \fancyLstar$.
    As above, $\op{p}{G;\theta} = \op{p}{N;A}$
    whence $\op{p}{M} \leq \op{p}{N;A}$.
    Lemma~\ref{f:2}(a) implies $\op{p}{M}' \leq \sol{N}$.
    Consequently $\layerr{N} \leq \theta(\nn{\op{p}{M}'}) \leq M$.
    As $S \leq N$ and $\ff{N}$ is $\theta(A)$-invariant and solvable
    we  have $\ff{N} \leq \ff{S}$.
    Then $\primes{\ff{N}} \subseteq \primes{\ff{S}} \subseteq \primes{\ff{M}} = \listset{p}$.
    Theorem~\ref{ben:2}(c) and the First Uniqueness Theorem imply $M_{a} = N$,
    so $S \leq M$.

    (b). Recall that $S$ contains every $\theta(A)$-invariant solvable $\theta$-subgroup.
    Using (a) we have $\gfitt{M_{a}} = \ff{M_{a}} \leq S \leq M_{b}$
    so $M_{a} \linksto M_{b}$.
    Similarly $M_{b} \linksto M_{a}$.
    Theorem~\ref{ben:2}(c) and the First Uniqueness Theorem
    force $M_{b} = M_{a}$.
\end{proof}

\begin{UnnumberedRemark}
    In \cite{PFpp} it is conjectured that if $p$ is a prime then to each
    nontrivial $p$-group $P$ there exists a nontrivial characteristic
    subgroup $W(P)$ such that whenever $A$ acts coprimely on the group $M$
    and $M$ has characteristic $p$ then $W(\op{p}{M;A}) \normal M$.
    A proof of this conjecture would lead to a much cleaner proof of the
    Second Uniqueness Theorem.
    The conjecture is known to be true if $p>3$,
    see \cite{PFezj}.
\end{UnnumberedRemark}

\section{$A$-components} \label{acomp}
For each $a \in A\nonid$ let \[
    \Omega_{a} = \set{K}{ \text{$K \in \comp{A}{M}$ for some $M$ with $\theta(a) \leq M \in \fancyL$}}.
\]
In particular, $\comp{A}{\theta(a)} \cup \comp{A}{M_{a}} \subseteq \Omega_{a}$.
Let \[
    \Omega = \bigcup_{a \in A\nonid} \Omega_{a}.
\]
The Second Uniqueness Theorem and Lemma~\ref{sm:2} imply \[
    \Omega \not= \emptyset.
\]
The subsequent analysis is dominated by the elements of $\Omega$.
Recall the definition of $\cstar{K}{A}$ given in \S\ref{as}.

\begin{Theorem}\label{acomp:1}
    Let $K,L \in \Omega$.
    The following are equivalent:
    \begin{enumerate}
        \item[(a)]  $[K,L] \not= 1$.

        \item[(b)]  $\cstar{K}{A} = \cstar{L}{A}$.

        \item[(c)]  $[\cstar{K}{A}, \cstar{L}{A}] \not= 1$.
    \end{enumerate}
    In particular `does not commute' is an equivalence relation on $\Omega$.
\end{Theorem}

\begin{Lemma}\label{acomp:2}
    Suppose $a,b \in A\nonid, K \in \Omega_{a}, L \in \Omega_{b}$
    and $N \in \fancyL$.
    Set $K_{0} = \layerr{K \cap N}$ and $L_{0} = \layerr{L \cap N}$.
    Assume that $[K_{0},L_{0}] \not= 1$.
    Then there exists $X$ such that:
    \begin{enumerate}
        \item[(a)]  $\listgen{K_{0},L_{0}} \leq X \in \compasol{N}$.

        \item[(b)]  If $[K_{0},a] \not= 1$ then $X = K_{0}$ and
                    if $[L_{0},b] \not= 1$ then $X = L_{0}$.

        \item[(c)]  $\cstar{K}{A} = \cstar{L}{A}$.

        \item[(d)]  If $X$ is constrained then $K_{0} = L_{0}$
                    and $X = K_{0}\sol{X}$.
    \end{enumerate}
\end{Lemma}
\begin{proof}
    Choose $M \in \fancyL$ with $\theta(a) \leq M$ and $K \in \comp{A}{M}$.
    Now $\cz{N}{a} \leq \theta(a) \cap N \leq M \cap N$ so Hypothesis~\ref{aut:1}
    is satisfied with $N$ and $M \cap N$ in the roles of $G$ and $H$ respectively.
    Also $K \subnormal M$ so $K_{0} \subnormal M \cap N$.
    As $[K_{0},L_{0}] \not= 1$ we have $K_{0} \not= 1$ and then
    Lemma~\ref{as:4} implies $K_{0}$ is $A$-quasisimple,
    so $K_{0} \in \comp{A}{M \cap N}$.
    Theorem~\ref{aut:4} implies there exists $\widetilde{K}$ with \[
        K_{0} \leq \widetilde{K} \in \compasol{N}.
    \]
    Similarly there exists $\widetilde{L}$ with \[
        L_{0} \leq \widetilde{L} \in \compasol{N}.
    \]
    Since $[K_{0},L_{0}] \not= 1$ we have $[\widetilde{K},\widetilde{L}] \not= 1$.
    Consequently either $\widetilde{K}$ and $\widetilde{L}$ are both
    semisimple or both constrained.

    Suppose that $\widetilde{K}$ and $\widetilde{L}$ are both semisimple.
    Since $[\widetilde{K},\widetilde{L}] \not= 1$ this forces $\widetilde{K} = \widetilde{L}$.
    Put $X = \widetilde{K}$.
    Then (a) holds and (b) follows from Theorem~\ref{aut:4}(a).
    We claim that $\cstar{\widetilde{K}}{A} = \cstar{K_{0}}{A}$.
    If $[K_{0},a] \not= 1$ then (b) implies $\widetilde{K} = K_{0}$ and the claim is clear.
    Suppose $[K_{0},a] = 1$.
    Theorem~\ref{aut:4}(b) implies $K_{0} = \layerr{\cz{\widetilde{K}}{a}}$.
    In particular, $K_{0}$ is $\cz{\widetilde{K}}{A}$-invariant.
    The claim follows from Lemma~\ref{as:4}.
    Similarly $\cstar{\widetilde{L}}{A} = \cstar{L_{0}}{A}$.
    Also by Lemma~\ref{as:4},
    $\cstar{K}{A} = \cstar{K_{0}}{A}$ and $\cstar{L}{A} = \cstar{L_{0}}{A}$.
    Since $\widetilde{K} = \widetilde{L}$,
    (c) follows.
    Note that (d) is not applicable in this case.

    Suppose $\widetilde{K}$ and $\widetilde{L}$ are both constrained.
    Theorem~\ref{aut:4} implies \[
        [K_{0},a] = 1 \quad\text{and}\quad \widetilde{K} = K_{0}\sol{\widetilde{K}}.
    \]
    In particular $[\widetilde{K},a] \leq \sol{\widetilde{K}}$.
    Recall that $K_{0} \in \comp{A}{M \cap N}$.
    As $\theta(a) \leq M$ and $[K_{0},a] = 1$
    it follows that $K_{0} \in \comp{A}{\theta(a) \cap N}$.
    Similarly
    \begin{equation}\tag{$*$}
        L_{0} \in \comp{A}{\theta(b) \cap N}.
    \end{equation}
    Since $L_{0}$ is $A$-quasisimple,
    we have $[L_{0},a] = 1$ or $L_{0}$.
    Suppose $[L_{0},a] = L_{0}$.
    Then $[\widetilde{L},a] \not\leq \sol{\widetilde{L}}$
    and since every proper $A$-invariant normal subgroup of $\widetilde{L}$
    is solvable it follows that $\widetilde{L} = [\widetilde{L},a]$.
    As $[\widetilde{K},a] \leq \sol{\widetilde{K}}$ we have
    $\widetilde{K} \not= \widetilde{L}$.
    Theorem~\ref{aut:4}(d) implies $[\widetilde{K},\widetilde{L}] = 1$,
    a contradiction.
    We deduce that $[L_{0},a] = 1$.
    Then $(*)$ implies \[
        L_{0} \in \comp{A}{\theta(a) \cap \theta(b) \cap N}.
    \]
    Similarly $K_{0} \in \comp{A}{\theta(a) \cap \theta(b) \cap N}$
    so as $[K_{0},L_{0}] \not= 1$ we must have $K_{0} = L_{0}$.
    In particular,
    $\widetilde{K} \cap \widetilde{L}$ is nonsolvable
    so $\widetilde{K} = \widetilde{L}$.
    Put $X = \widetilde{K}$.
    Then (a) holds.
    (b) is not applicable in this case.
    Lemma~\ref{as:4} implies that
    $\cstar{K}{A} = \cstar{K_{0}}{A} = \cstar{L_{0}}{A} = \cstar{L}{A}$.
    Then (c) holds.
    (d) has also been proved.
\end{proof}

\begin{proof}[Proof of Theorem~\ref{acomp:1}]
    Suppose (a) holds, so $[K,L] \not= 1$.
    Lemma~\ref{as:5}(a) implies there exists $D \in \hyp{A}$ such that
    $\cz{L}{d}$ is $A$-quasisimple for all $d \in D\nonid$
    and $[K,\cz{L}{D}] \not= 1$.
    Lemma~\ref{as:5}(b) implies there exists $d \in D\nonid$ such that
    $\cz{K}{d}$ is $A$-quasisimple and $[\cz{K}{d}, \cz{L}{D}] \not= 1$.
    Then $[\cz{K}{d},\cz{L}{d}] \not= 1$.

    Put $N = \theta(d)$.
    Now $\cz{K}{d} = K \cap N$ so $\cz{K}{d} = \layerr{K \cap N}$.
    Similarly $\cz{L}{d} = \layerr{L \cap N}$.
    Lemma~\ref{acomp:2} implies $\cstar{K}{A} = \cstar{L}{A}$ so (b) holds.

    Lemma~\ref{as:4} implies that $\cstar{K}{A}$ is nonabelian.
    The remaining implications follow trivially.
\end{proof}

\section{The Balance Theorem} \label{b}
The aim of this section is to prove the following.

\begin{Theorem}[The Balance Theorem] \label{b:1}
    Suppose $a,b \in A\nonid, K \in \Omega_{a}$
    and $\layerr{K \cap M_{b}} \not= 1$.
    Then $\layerr{K \cap M_{b}}$ is $A$-quasisimple and is contained
    in an $A$-component of $M_{b}$.
    In particular \[
        \layerr{K \cap M_{b}} \leq \layerr{M_{b}}.
    \]
\end{Theorem}

\noindent A number of lemmas are required.
Recall that $S$ is the unique maximal $\theta(A)$-invariant
solvable $\theta$-subgroup.

\begin{Lemma} \label{b:2}
    Suppose $a \in A\nonid, K \in \Omega_{a}$ and $\theta$ is nearsolvable.
    Let $Y$ be a nonsolvable $A\cz{K}{A}$-invariant subgroup of $K$.
    Then \[
        K = \listgen{Y, K \cap S}.
    \]
\end{Lemma}
\begin{proof}
    Choose $M$ with $\theta(a) \leq M \in \fancyL$ and $K \in \comp{A}{M}$.
    Since $\theta(A) \leq M$ we have $S \cap M = \op{\mathrm{sol}}{M;A}$.
    Since $K$ is an $A$-invariant subnormal subgroup of $M$,
    Corollary~\ref{psgps:4} implies
    $\op{\mathrm{sol}}{K;A} = K \cap \op{\mathrm{sol}}{M;A}$.
    Then $\op{\mathrm{sol}}{K;A} = K \cap S$.
    Now $\theta$ is nearsolvable so $\theta(A)$ and hence $\cz{K}{A}$ is solvable.
    Apply Theorem~\ref{as:1}(e) to $K/\zz{K}$.
\end{proof}

\begin{Lemma}\label{b:3}
    Suppose $a,b \in A\nonid, K \in \Omega_{a}$ and $\layerr{K \cap M_{b}} \not= 1$.
    Then $\layerr{K \cap M_{b}}$ is $A$-quasisimple.
    Suppose also that $\layerr{K \cap M_{b}}$
    is not contained in an $A$-component of $M_{b}$.
    Then:
    \begin{enumerate}
        \item[(a)]  $K \leq M_{b}$ and there exists $\widetilde{K}$
                    with \[
                        K \leq \widetilde{K} \in \compasol{M_{b}}.
                    \]

        \item[(b)]  $\widetilde{K}$ is constrained,
                    $\widetilde{K} = K\sol{\widetilde{K}}$
                    and $K \cap \layerr{M_{b}} \leq \zz{K}$.

        \item[(c)]  If $c \in A \setminus \cz{A}{K}$ then
                    $\widetilde{K} = [\widetilde{K},c] = \listgen{K, \cz{\sol{\widetilde{K}}}{c}}$.

        \item[(d)]  $\theta$ is nearsolvable.
    \end{enumerate}
\end{Lemma}
\begin{proof}
    Choose $M$ with $\theta(a) \leq M \in \fancyL$ and $K \in \comp{A}{M}$.
    Set $K_{0} = \layerr{K \cap M_{b}}$.
    Lemma~\ref{as:4} implies that $K_{0}$ is $A$-quasisimple.
    Since $\cz{M_{b}}{a} \leq \theta(a) \cap M_{b} \leq M \cap M_{b}$,
    Hypothesis~\ref{aut:1} is satisfied with $M_{b}$ and $M \cap M_{b}$
    in the roles of $G$ and $H$ respectively.
    As $K \cap M_{b} \subnormal M \cap M_{b}$ we have $K_{0} \in \comp{A}{M \cap M_{b}}$.
    Theorem~\ref{aut:4} implies that there exists $\widetilde{K}$ with \[
        K_{0} \leq \widetilde{K} \in \compasol{M_{b}}.
    \]
    By assumption,
    $K_{0}$ is not contained in an $A$-component of $M_{b}$ so
    $\widetilde{K}$ is constrained.
    Then $1 \not= \sol{\widetilde{K}} \leq \sol{M_{b}} \in \fancyL$
    and McBride's Dichotomy implies that $\theta$ is nearsolvable.

    Lemma~\ref{as:4} implies that $\cstar{K_{0}}{A} = \cstar{K}{A}$.
    Since $\widetilde{K}$ is constrained we have $[\widetilde{K}, \layerr{M_{b}}] = 1$,
    whence \[
        [\cstar{K}{A}, \layerr{M_{b}}] = 1.
    \]
    Recall that $\layerr{M_{b}}$ is generated by the $A$-components of $M_{b}$.
    Theorem~\ref{acomp:1} implies $[K,\layerr{M_{b}}] = 1$.
    If $\layerr{M_{b}} \not= 1$ then $K \leq M_{b}$.
    If $\layerr{M_{b}} = 1$ then the Second Uniqueness Theorem implies $S \leq M_{b}$
    and then Lemma~\ref{b:2} yields $K = \listgen{K_{0}, K \cap S} \leq M_{b}$.
    In both cases,
    $K \leq M_{b}$ so $K_{0} = K$.

    Also $K \cap \layerr{M_{b}} \leq \widetilde{K} \cap \layerr{M_{b}} \leq \sol{\widetilde{K}}$
    so  $K \cap \layerr{M_{b}}$ is solvable normal subgroup of the $A$-quasisimple group $K$.
    Hence it is contained in $\zz{K}$.

    To prove (c), suppose $c \in A \setminus \cz{A}{K}$.
    Then $K = [K,c] \leq [\widetilde{K},c]$ so as $\widetilde{K} \in \compasol{M_{b}}$
    it follows that $\widetilde{K} = [\widetilde{K},c]$.
    The remaining assertion follows from Theorem~\ref{aut:4}(c).
\end{proof}

\begin{proof}[Proof of the Balance Theorem]
    Assume false.
    Lemma~\ref{b:3} implies there exists a constrained $\widetilde{K}$
    with \[
        K \leq \widetilde{K} \in \compasol{M_{b}}.
    \]
    We may suppose that $(a,b,K)$ has been chosen to maximize $\widetilde{K}$.
    \setcounter{Claim}{0}
    \begin{Claim} \label{b:Claim1}
        Let $L \in \Omega$ and suppose $[K,L] \not= 1$.
        Then $K = L$.
    \end{Claim}
    \begin{proof}
        Lemma~\ref{as:5} implies there exists $D \in \hyp{A}$
        and $d \in D\nonid$ such that $\cz{K}{d}$ and $\cz{L}{d}$ are $A$-quasisimple.
        Let $K_{0} = \layerr{K \cap M_{d}}$ and $L_{0} = \layerr{L \cap M_{d}}$.
        Since $\cz{K}{d} \leq K \cap \theta(d) \leq K \cap M_{d}$,
        Lemma~\ref{as:4} implies $\cz{K}{d} \leq K_{0}$.
        Similarly $\cz{L}{d} \leq L_{0}$.
        Now $[K,L] \not= 1$ so Theorem~\ref{acomp:1}
        implies $[\cstar{K}{A}, \cstar{L}{A}] \not= 1$.
        Then $[K_{0},L_{0}] \not= 1$.

        Lemma~\ref{acomp:2},
        with $M_{d}$ in the role of $N$,
        implies there exists $X$ with \[
            \listgen{K_{0},L_{0}} \leq X \in \compasol{M_{d}}.
        \]
        Suppose $X$ is constrained.
        Then $K_{0}$ is not contained in a component of $M_{d}$.
        As $K_{0} = \layerr{K \cap M_{d}}$,
        Lemma~\ref{b:3}(a) implies $K \leq M_{d}$,
        so $K_{0} = K$.
        Similarly $L_{0} = L$.
        Lemma~\ref{acomp:2}(d) implies $K_{0} = L_{0}$
        so we are done in this case.

        Suppose $X$ is semisimple.
        Then $X$ is $A$-quasisimple.
        Now $\cz{K}{d} \leq X \cap M_{b}$ so as $\cz{K}{d}$ is $A$-quasisimple,
        Lemma~\ref{as:4}(b),
        with $X$ and $X \cap M_{b}$ in the roles of $K$ and $H$ respectively,
        implies $\cz{K}{d} \leq \layerr{X \cap M_{b}}$.
        By Lemma~\ref{b:3},
        $K \cap \layerr{M_{b}} \leq \zz{K}$ so $\layerr{X \cap M_{b}} \not\leq \layerr{M_{b}}$.
        In particular, $\layerr{X \cap M_{b}}$ is not contained in an $A$-component of $M_{b}$.
        Now $X \in \comp{A}{M_{d}} \subseteq \Omega_{d}$ so
        Lemma~\ref{b:3}(a) forces $X \leq M_{b}$ and $X \cap \layerr{M_{b}} \leq \zz{X}$.

        Now $\cz{L}{d}$ is $A$-quasisimple and $\cz{L}{d} \leq X \leq M_{b}$.
        Then $\cz{L}{d} \leq \layerr{L \cap M_{b}} \not\leq \layerr{M_{b}}$
        and Lemma~\ref{b:2} forces $L \leq M_{b}$.
        We apply Lemma~\ref{acomp:2},
        with $M_{b},K$ and $L$ in the roles of $N,K_{0}$ and $L_{0}$ respectively.
        Since $K$ is not contained in an $A$-component of $M_{b}$,
        Lemma~\ref{acomp:2}(d) forces $K = L$.
    \end{proof}
    \begin{Claim} \label{b:Claim2}
        Suppose $c \in A \setminus \cz{A}{K}$ and $\layerr{M_{c}} \not= 1$.
        Then $\widetilde{K} \in \compasol{M_{c}}$.
    \end{Claim}
    \begin{proof}
        Claim~\ref{b:Claim1} implies $K$ normalizes $\layerr{M_{c}}$
        so $K \leq M_{c}$.
        Using Lemma~\ref{b:3}(c) we have \[
            \widetilde{K} = \listgen{K, \cz{\sol{\widetilde{K}}}{c}} %
            \leq \listgen{K, \theta(c)} \leq M_{c}.
        \]
        As previously, Theorem~\ref{aut:4},
        implies that there exists $K^{*}$ with \[
            K \leq K^{*} \in \compasol{M_{c}}.
        \]
        Then $K \leq K^{*} \cap \widetilde{K} \subnormal \widetilde{K}$.
        But $\widetilde{K}$ contains no proper $A$-invariant nonsolvable subnormal subgroups,
        whence $K^{*} \cap \widetilde{K} = \widetilde{K}$ and $\widetilde{K} \leq K^{*}$.
        As $\widetilde{K}$ is constrained we have $K < \widetilde{K} \leq K^{*}$
        and then Claim~\ref{b:Claim1} implies $K^{*} \not\in \Omega$,
        so $K^{*}$ is not semisimple.
        The maximal choice of $\widetilde{K}$ forces $\widetilde{K} = K^{*}$,
        which proves the claim.
    \end{proof}

    Choose $N$ with \[
        \theta(\nn{\widetilde{K}}) \leq N \in \Theta^{*}.
    \]

    \begin{Claim} \label{b:Claim3}
        Suppose $c \in A \setminus \cz{A}{K}$ and $\layerr{M_{c}} \not= 1$.
        Then $M_{c} = N$.
    \end{Claim}
    \begin{proof}
        Claim~\ref{b:Claim2} implies $\widetilde{K} \in \compasol{M_{c}}$.
        In particular, $\widetilde{K} \normal \ostar{M_{c}}$
        and then $\ostar{M_{c}} \leq N$.
        Lemma~\ref{b:3} implies $\widetilde{K} = [\widetilde{K},c]$
        so using Theorem~\ref{aut:3}(b) we have \[
            1 \not= \widetilde{K} \leq [\ostar{M_{c}},c]^{(\infty)} \normal \ostar{N}.
        \]
        Theorem~\ref{sm:1} implies $\ostar{N} \leq M_{c}$ and then
        Corollary~\ref{m:2} forces $M_{c} = N$.
    \end{proof}

    It is straightforward to complete the proof of the Balance Theorem.
    Now $\theta$ is nearsolvable so $\theta(A)$ is solvable and we
    may choose $B$ with $\cz{A}{K} \leq B \in \hyp{A}$.
    The Second Uniqueness Theorem and Claim~\ref{b:Claim3} implies that
    $M_{c}$ takes at most two values as $c$ ranges over $A \setminus B$.
    Lemma~\ref{sm:2}(b) supplies a contradiction.
\end{proof}

\section{The Structure Theorem} \label{s}
The following result will be proved.
Once it has,
Theorem~\ref{as:2} and Lemma~\ref{psgps:6} will
supply a contradiction and complete the proof of
the Signalizer Functor Theorem.

\begin{Theorem}[The Structure Theorem]  \label{s:1} \mbox{}
    \begin{enumerate}
        \item[(a)]  If $a \in A\nonid$ with $\layerr{\theta(a)} \not= 1$ then
                    $\layerr{\theta(a)}$ is $A$-simple,
                    $\ff{\theta(a)} = 1$ and $\cz{A}{\layerr{\theta(a)}} = \listgen{a}$.

        \item[(b)]  For all $a,b \in A\nonid$, \[
                        \layerr{\layerr{\theta(a)} \cap \cc{b}} \leq \layerr{\theta(b)}.
                    \]

        \item[(c)]  $G = \gen{ \layerr{\theta(a)} }{ a \in A\nonid }$.
    \end{enumerate}
\end{Theorem}

\begin{Lemma}\label{s:2}
    Let $B \in \hyp{A}$.
    Then $G = \gen{\layerr{M_{b}}}{b \in B\nonid}$.
\end{Lemma}
\begin{proof}
    For each $D \leq A$ set $G_{D} = \gen{ \layerr{M_{d}} }{ d \in D\nonid}$.
    Let $D \in \hyp{A}$ and $a \in A\nonid$.
    We claim that $\layerr{M_{a}} \leq G_{D}$.
    Indeed, let $K \in \comp{A}{M_{a}}$.
    If $d \in D\nonid$ and $\cz{K}{d}$ is $A$-quasisimple then
    $\cz{K}{d} \leq \layerr{K \cap M_{d}}$ by Lemma~\ref{as:4}.
    Lemma~\ref{as:5} and the Balance Theorem yield
    \begin{align*}
        K   &= \gen{ \cz{K}{d} }{ \text{$d \in D\nonid$ and $\cz{K}{d}$ is $A$-quasisimple} }\\
            &\leq \gen{ \layerr{M_{d}} }{ d \in D\nonid} \leq G_{D}.
    \end{align*}
    The claim is established.
    In particular, $G_{A} = G_{D}$.
    Note that $\theta(D)$ normalizes $G_{D}$ and hence $G_{A}$.
    Using Lemma~\ref{er:5} and the fact that $G = \gen{\theta(a)}{a \in A\nonid}$
    we have $G = \gen{\theta(D)}{D \in \hyp{A}}$
    and it follows that $G_{A} \normal G$.
    We have previously observed that $\layerr{M_{a}} \not= 1$ for some $a$
    by the Second Uniqueness Theorem.
    Hence $G_{A} = G$.
    Then $G = G_{B}$.
\end{proof}

\begin{Lemma}\label{s:3}
    \begin{enumerate}
        \item[(a)]  Let $K,L \in \Omega$.
                    Then $[K,L] \not= 1$ and $\cstar{K}{A} = \cstar{L}{A}$.

        \item[(b)]  Let $a \in A\nonid$ with $\layerr{M_{a}} \not= 1$.
                    Then $\layerr{M_{a}}$ is $A$-quasisimple.
    \end{enumerate}
\end{Lemma}
\begin{proof}
    (a). Theorem~\ref{acomp:1} implies that `does not commute' is an equivalence
    relation on $\Omega$.
    Let $K_{1},\ldots,K_{n}$ be representatives for the equivalence classes
    and let $[K_{i}]$ denote the subgroup generated by the class of $K_{i}$.
    Suppose $n \geq 2$.
    Then for all $i \not= j$ we have $[K_{i}] \leq \theta(\cc{K_{j}})$
    so $[K_{i}]$ is a $\theta$-subgroup.
    Moreover $[K_{i}]$ and $[K_{j}]$ commute.
    Then $[K_{i}] \normal \listgen{\Omega}$.
    Lemma~\ref{s:2} implies $\listgen{\Omega} = G$
    and then Lemma~\ref{m:1}(d) supplies a contradiction.
    Hence there is only one equivalence class so $[K,L] \not= 1$.
    Theorem~\ref{acomp:1} implies $\cstar{K}{A} = \cstar{L}{A}$.

    (b). Distinct $A$-components of $M_{a}$ commute.
    Then (a) implies $\layerr{M_{a}}$ is $A$-quasisimple.
\end{proof}

\begin{Lemma}\label{s:4}
    Let $a \in A\nonid$.
    Then $\layerr{M_{a}} = 1$ or $\ff{M_{a}} = 1$.
\end{Lemma}
\begin{proof}
    Assume false.
    Then $\ff{M_{a}} \not= 1$ and McBride's Dichotomy implies that $\theta$ is nearsolvable,
    whence $\theta(A)$ is solvable.
    By Coprime Action(d) there exists $B \in \hyp{A}$ with $\cz{\ff{M_{a}}}{B} \not= 1$.
    Set $Z = \cz{\ff{M_{a}}}{B}$.
    Let $b \in B\nonid$ and suppose $\layerr{M_{b}} \not= 1$.
    Lemma~\ref{s:3} implies $\cstar{\layerr{M_{a}}}{A} = \cstar{\layerr{M_{b}}}{A}$
    so as $[Z, \layerr{M_{a}}] = 1$ we have $[Z, \cstar{\layerr{M_{b}}}{A}] = 1$.
    Now $\theta$ is nearsolvable so $\cstar{\layerr{M_{b}}}{A} = \cz{\layerr{M_{b}}}{A}$
    and $Z \leq \theta(b) \leq M_{b}$.
    Lemma~\ref{as:6} implies $[Z,\layerr{M_{b}}] = 1$.
    But then Lemma~\ref{s:2} implies $Z \normal G$,
    contrary to Lemma~\ref{m:1}(d).
\end{proof}

\begin{Lemma}\label{s:5}
    Let $a \in A\nonid$ with $\layerr{M_{a}} \not= 1$.
    Then \[
        M_{a} = \theta(a).
    \]
\end{Lemma}
\begin{proof}
    Assume false.
    Recall that $\theta(a) = \cz{M_{a}}{a}$.
    Lemma~\ref{s:4} implies $\ff{M_{a}} = 1$ so
    Coprime Action(g) implies $[\layerr{M_{a}},a] \not= 1$.
    Set $K = \layerr{M_{a}}$,
    so $K \in \Omega$ by Lemma~\ref{s:3}(b).
    Lemma~\ref{as:5}(a) implies there exists $D \in \hyp{A}$ such that
    $[\cz{K}{D},a] \not= 1$ and $\cz{K}{d}$ is $A$-quasisimple for all $d \in D\nonid$.

    Let $L \in \Omega$.
    Lemma~\ref{s:3}(a) and Theorem~\ref{acomp:1} imply $[\cstar{K}{A},\cstar{L}{A}] \not= 1$.
    Let $d \in D\nonid$ and suppose $\cz{L}{d}$ is $A$-quasisimple.
    Note that $[\cz{K}{d},a] \not= 1$ and that
    $1 \not= [\cstar{K}{A},\cstar{L}{A}] \leq [\cz{K}{d},\cz{L}{d}]$.
    Lemma~\ref{acomp:2}(b),
    with $\theta(d), \cz{K}{d}$ and $\cz{L}{d}$
    in the roles of $N,K_{0}$ and $L_{0}$ respectively,
    forces $\cz{L}{d} \leq \cz{K}{d} \leq K$.
    Then Lemma~\ref{as:5}(b) implies $L \leq K$.
    But then, by Lemma~\ref{s:2},
    $G = \listgen{\Omega} \leq K$,
    a contradiction.
\end{proof}

\begin{proof}[Proof of the Structure Theorem]
(a). The Balance Theorem implies $\layerr{\theta(a)} \leq \layerr{M_{a}}$
so $\layerr{M_{a}} \not= 1$.
Lemma~\ref{s:5} implies $M_{a} = \theta(a)$.
Lemmas~\ref{s:4} and \ref{s:3}(b) imply that $\ff{\theta(a)} = 1$
and then that $\layerr{\theta(a)}$ is $A$-simple.
Let $B = \cz{A}{\layerr{\theta(a)}}$.
Coprime Action(g) implies $B = \cz{A}{\theta(a)}$,
so as $\theta(a) = M_{a} \in \fancyLstar$ we have $M_{a} \leq \theta(b) \leq M_{b}$
and then $M_{a} = M_{b}$ for all $b \in B\nonid$.
Recall that $\rank{A} = 3$.
Then Lemma~\ref{sm:2} implies $B$ is cyclic.
Consequently $B = \listgen{a}$ as required.

(b). Set $J = \layerr{\layerr{\theta(a)} \cap \cc{b}} \leq \theta(b)$
and $H = \layerr{\theta(a)} \cap M_{b}$.
We may assume that $J \not= 1$.
Then $\layerr{\theta(a)} \not= 1$.
(a) implies $\layerr{\theta(a)}$ is $A$-simple,
so $\layerr{\theta(a)} \in \Omega$.
Since $J \leq H$,
Lemma~\ref{as:4} implies $H^{(\infty)}$ is $A$-quasisimple.
The Balance Theorem implies $H^{(\infty)} \leq \layerr{M_{b}}$,
whence $M_{b} = \theta(b)$ by Lemma~\ref{s:5}.
Since $J = J^{(\infty)} \leq H^{(\infty)}$ we have $J \leq \layerr{\theta(b)}$.

Finally, (c) follows from Lemma~\ref{s:2} and Lemma~\ref{s:5}.
This completes the proof of the Structure Theorem and hence of the
Signalizer Functor Theorem.
\end{proof}

\bibliographystyle{amsplain}

\end{document}